\documentclass[12pt]{amsart}

\usepackage[headings]{fullpage}
\usepackage{amsmath}
\usepackage{amsthm,amsfonts,amssymb,amscd}
\usepackage{lastpage}
\usepackage{enumerate}
\usepackage[mathscr]{eucal}
\usepackage{mathrsfs}
\usepackage[dvipsnames]{xcolor}
\usepackage{graphicx}
\usepackage{listings}
\usepackage{hyperref}
\usepackage{array}
\usepackage[utf8]{inputenc}
\usepackage[T1]{fontenc}
\usepackage{tikz}
\usepackage{tensor}
\usepackage{caption}
\usepackage{subcaption}
\usepackage[toc,page]{appendix}
\usepackage{datetime}
\usepackage{verbatim}
\usepackage{mathtools}
\usepackage[draft]{todonotes}

\def\XXint#1#2#3{{\setbox0=\hbox{$#1{#2#3}{\int}$ }
\vcenter{\hbox{$#2#3$ }}\kern-.6\wd0}}

\numberwithin{equation}{section}

\theoremstyle{plain}
\newtheorem{thm}{Theorem}[section]
\newtheorem{lemma}[thm]{Lemma}

\newtheorem{prop}[thm]{Proposition}

\newtheorem{definition}[thm]{Definition}

\newtheorem{remark}[thm]{Remark}

\allowdisplaybreaks

\begin{document}

\author[D. Wang]{Dinghuai Wang}
\address{Dinghuai Wang:
School of Mathematics and Statistics \\
Anhui Normal University \\
Wuhu 241002 \\
People's Republic of China}
\email{Wangdh1990@126.com}


\title{Caffarelli-Kohn-Nirenberg Inequalities in Weak Lebesgue Spaces}

\keywords{Fractional integrals, sparse operators, Calder\'on-Zygmund decomposition, Hardy inequality, Caffarelli-Kohn-Nirenberg inequality \\
\indent{{\it {2020 Mathematics Subject Classification.}}} Primary 42B20, 42B37. Secondary 42B25.}

\begin{abstract}
By employing harmonic analysis techniques,
we derive weak-type Caffarelli-Kohn-Nirenberg inequalities under natural parameter conditions. A key feature of these weak-type versions is that they remain valid even at critical parameter values where the classical inequalities fail. As an important corollary, we obtain weak-type Hardy inequalities that hold true even in the critical dimension \(d = p\). The methods developed here are sufficiently flexible to handle homogeneous, non-homogeneous and anisotropic weights, providing a unified approach to various endpoint cases in interpolation theory.

\end{abstract}

\maketitle


\newcommand{\N}{\mathbb{N}}
\newcommand{\C}{\mathbb{C}}
\newcommand{\R}{\mathbb{R}}
\newcommand{\Z}{\mathbb{Z}}
\newcommand{\Q}{\mathbb{Q}}
\newcommand{\Primes}{\mathbb{P}}
\newcommand{\csigma}{\mathfrak{S}}
\newcommand{\1}{\mathbbm{1}}
\newcommand{\bb}[1]{\mathbb{#1}}
\newcommand{\mcal}[1]{\mathcal{#1}}
\newcommand{\T}{\mathbb{T}}
\newcommand{\A}{\mathbb{A}}
\newcommand{\di}{\text{div}}
\newcommand{\supp}{\text{supp}}
\newcommand{\I}{\mathcal{I}}
\newcommand{\D}{\mathcal{D}}
\newcommand{\Des}{\text{Des}}
\renewcommand{\S}{\mcal{S}}

\section{Introduction}

The Caffarelli-Kohn-Nirenberg (CKN) interpolation inequality is a fundamental result in the theory of Sobolev spaces, with extensive applications in partial differential equations and calculus of variations. First introduced in \cite{CKN1982}, this inequality generalizes both the classical Sobolev inequality and the Hardy inequality, providing a unified framework for interpolation estimates with power-law weights.

For $d\geq 1$, let $s, p, q, \gamma_1, \gamma_2, \gamma_3$ and $\theta$ satisfy
\begin{equation}\label{CKN-1}
s>0, p,q\geq 1, 0\leq \theta\leq 1,
\end{equation}
\begin{equation}\label{CKN-2}
\frac{1}{s}+\frac{\gamma_{1}}{d}>0, \frac{1}{p}+\frac{\gamma_{2}}{d}>0, \frac{1}{q}+\frac{\gamma_{3}}{d}>0,
\end{equation}
\begin{equation}\label{CKN-3}
\frac{1}{s}+\frac{\gamma_{1}}{d}=\theta\Big(\frac{1}{p}+\frac{\gamma_{2}-1}{d}\Big)+(1-\theta)\Big(\frac{1}{q}+\frac{\gamma_{3}}{d}\Big),
\end{equation}
\begin{equation}\label{CKN-4}
\gamma_{1}\leq \theta\gamma_{2}+(1-\theta)\gamma_{3},
\end{equation}
\begin{equation}\label{CKN-5}
\frac{1}{s}\leq \frac{\theta}{p}+\frac{1-\theta}{q} \quad \text{if} \quad \theta=0 \quad \text{or} \quad \theta=1 \quad \text{or} \quad
\frac{1}{s}+\frac{\gamma_{1}}{d}=\frac{1}{p}+\frac{\gamma_{2}-1}{d}=\frac{1}{q}+\frac{\gamma_{3}}{d}.
\end{equation}

\begin{thm}[Caffarelli, Kohn, Nirenberg \cite{CKN1982}]\label{thm:CKN}
For $d\geq 1$, let $s, p,q, \gamma_{1},\gamma_{2},\gamma_{3}$ and $\theta$ satisfy \eqref{CKN-1} and \eqref{CKN-2}. Then
there exists a positive constant $C$ such that
\begin{equation}\label{CKN}
\||x|^{\gamma_{1}}f\|_{L^{s}(\R^d)}\leq C\||x|^{\gamma_{2}}\nabla f\|^{\theta}_{L^{p}(\R^d)}\||x|^{\gamma_{3}}f\|^{1-\theta}_{L^{q}(\R^d)}
\end{equation}
holds for all $f\in C^1_{c}(\R^d)$ if and only if \eqref{CKN-3}-\eqref{CKN-5} hold.
\end{thm}

Given (\ref{CKN-1}), condition (\ref{CKN-2}) holds if and only if $\||x|^{\gamma_1}f\|_{L^s(\R^d)}$, $\||x|^{\gamma_2}\nabla f\|_{L^p(\R^d)}$ and $\||x|^{\gamma_3}f\|_{L^q(\R^d)}$ are finite for all $f\in C_c^{\infty}(\mathbb{R}^n)$.
Recently, Li and Yan \cite{LY2023} established a more general anisotropic interpolation inequality that extends the range of the parameter $q$ in the (\ref{CKN}) inequalities from $q \geq 1$ to $q > 0$.

This inequality has seen numerous generalizations and extensions, including higher-order derivatives \cite{Lin}, cylindrical Sobolev-Hardy inequalities \cite{BT}, refined Hardy inequalities \cite{BCG1, BCG2}, and fractional Sobolev space settings \cite{DS, NS, NS2}. The study of best constants and extremal functions has also been a major focus \cite{CW, DEL}.
  Sharp Sobolev and isoperimetric inequalities with monomial weights, and related problems, are studied by Cabr\'{e}, Ros-Oton and Serra, see \cite{CR, CRS}.

There has been growing interest in weak-type versions of classical inequalities, particularly in endpoint cases where strong-type inequalities fail. For instance, the classical Hardy inequality
\[
\Big\|\frac{f}{|x|}\Big\|_{L^{p}(\R^d)} \lesssim \|\nabla f\|_{L^{p}(\R^d)}, \quad 1 < p < d,
\]
fails when $p = d$, but a weak-type analogue might still hold. Similarly, the Sobolev embedding
\[
\|f\|_{L^{q}(\R^d)} \lesssim \|\nabla f\|_{L^{p}(\R^d)}, \quad q = \frac{dp}{d-p},
\]
breaks down at the endpoint $p = d$, where one obtains the weaker BMO estimate instead.

In this paper, we establish weak-type versions of the CKN inequalities, where either the source or target Lebesgue norms are replaced by their weak counterparts. Our approach employs modern harmonic analysis techniques, particularly the method of domination by sparse operators, which has proven effective for proving two-weight inequalities \cite{Hytonen, Lerner, LernerOmbrosi}.

\subsection{Main results}

Let $d \ge 1$ and let $s, p, q, \gamma_1, \gamma_2, \gamma_3, \theta$ be real numbers satisfying
\begin{equation}\label{WCKN-1}
s > 0,\quad p \ge 1,\quad q > 0,\quad 0 \le \theta \le 1,
\end{equation}
\begin{equation}\label{WCKN-2}
    \frac{1}{p} + \frac{\gamma_2}{d} < 1 \quad \text{for } p > 1; \qquad
    \gamma_2 \le 0 \quad \text{for } p = 1,
\end{equation}
\begin{equation}\label{WCKN-2-0}
 \frac{1}{p} + \frac{\gamma_2-1}{d}\geq 0.
\end{equation}

Our main result is the following weak-type CKN inequality:

\begin{thm}\label{thm:main}
Let $d \ge 1$ and suppose $s, p, q, \gamma_1, \gamma_2, \gamma_3, \theta$ satisfy \eqref{WCKN-1}, \eqref{WCKN-2} and \eqref{WCKN-2-0}. Then there exists a constant $C> 0$ such that
\begin{equation}\label{WCKN-weak}
\big\|\,|x|^{\gamma_1} f\,\big\|_{L^{s,\infty}(\R^d)}
\le C\,\big\|\,|x|^{\gamma_2}\nabla f\,\big\|_{L^{p}(\R^d)}^{\theta}
\,\big\|\,|x|^{\gamma_3} f\,\big\|_{L^{q,\infty}(\R^d)}^{1-\theta}
\end{equation}
holds for all $f \in C_c^1(\R^d)$ if and only if conditions \eqref{CKN-3}--\eqref{CKN-4} are satisfied.
\end{thm}

\begin{remark}
The condition \eqref{WCKN-2} originates from our application of sparse domination techniques and Muckenhoupt weights. However, unlike the restrictions imposed on
$\gamma_1$ and $\gamma_3$ in inequality \eqref{CKN-2}, no further constraints on these parameters are needed in our setting.
\end{remark}

\begin{remark}
For  $\theta = 1$, the combination of \eqref{CKN-2} and \eqref{CKN-3} implies
$$\frac{1}{p}+\frac{\gamma_2-1}{d} > 0,$$
which is a stronger statement than the corresponding condition in \eqref{WCKN-2-0}.
\end{remark}

\begin{remark}
Unlike the classical CKN inequality, the condition \eqref{CKN-5} is not necessary.

(1) \text{Case $\theta = 0$.}
    Suppose the scaling condition
    \[
    \frac{1}{s}+\frac{\gamma_{1}}{d} = \frac{1}{q}+\frac{\gamma_{3}}{d}
    \]
    holds with $\gamma_3 > \gamma_1$. Applying Hölder's inequality in weak Lebesgue spaces yields
    \begin{equation}\label{Holder-WL}
    \big\||x|^{\gamma_1}f\big\|_{L^{s,\infty}(\R^d)}
    \leq \big\||x|^{\gamma_3}f\big\|_{L^{q,\infty}(\R^d)}
    \big\||x|^{\gamma_1-\gamma_3}\big\|_{L^{\frac{d}{\gamma_3-\gamma_1},\infty}(\R^d)}
    \lesssim \big\||x|^{\gamma_3}f\big\|_{L^{q,\infty}(\R^d)},
    \end{equation}
    where $\frac{1}{s} = \frac{1}{q} + \frac{\gamma_3 - \gamma_1}{d}$ (hence $\frac{1}{s} > \frac{1}{q}$).
    Inequality \eqref{Holder-WL} implies that the weak-type estimate \eqref{WCKN-weak} holds for $\theta = 0$ without the need for condition \eqref{CKN-5}.

(2) \text{Full scaling case.}
    Now assume the triple scaling condition
    \[
    \frac{1}{s}+\frac{\gamma_{1}}{d} = \frac{1}{p}+\frac{\gamma_{2}-1}{d} = \frac{1}{q}+\frac{\gamma_{3}}{d}
    \]
    holds and that \eqref{WCKN-weak} is valid for some $\theta \in \bigl(0, \frac{1}{s+1}\bigr)$.
    For any $\breve{q} > q$, define $\breve{\gamma_3}$ by
    \[
    \frac{1}{q}+\frac{\gamma_3}{d} = \frac{1}{\breve{q}}+\frac{\breve{\gamma_3}}{d}.
    \]
    Utilizing \eqref{Holder-WL}, we find that \eqref{WCKN-weak} remains valid for the parameters $s, p, \breve{q}, \gamma_1, \gamma_2, \breve{\gamma_3}, \theta$.
    If we take $\breve{q}$ sufficiently large, we get
\[
\frac{1}{s} > \frac{\theta}{p} + \frac{1-\theta}{\breve{q}}.
\]
This demonstrates that under the full scaling condition, the requirement \eqref{CKN-5} is not necessary.

(3) \text{Case $\theta = 1$.}
By an analogous argument, for any $\hat{s} < \min\{s, p\}$ satisfying
\[
\frac{1}{s}+\frac{\gamma_1}{d} = \frac{1}{\hat{s}}+\frac{\hat{\gamma}_1}{d},
\]
the inequality \eqref{WCKN-weak} continues to hold with the parameters $(\hat{s}, p, q, \hat{\gamma}_1, \gamma_2, \gamma_3, \theta)$.
\end{remark}

\begin{remark}
When $\theta = 1$, $s = p$, $\gamma_1 = -1$, and $\gamma_2 = 0$, Theorem \ref{thm:main} yields the weak-type Hardy inequality
\[
\Big\|\frac{f}{|x|}\Big\|_{L^{p,\infty}(\R^d)} \lesssim \|\nabla f\|_{L^{p}(\R^d)}.
\]
Notably, this inequality remains valid even in the critical case $p = d$, whereas the classical Hardy inequality fails for $p = d$.
\end{remark}

\begin{remark}
It is well known that the classical Sobolev embedding theorem states that for \( 1 < p < d \),
\[
\|f\|_{L^{q}(\R^d)} \lesssim \|\nabla f\|_{L^{p}(\R^d)},
\]
where \( q = \frac{dp}{d-p} \). When \( p = d \), the limiting case \( q = \infty \) fails, and one can only obtain the weaker estimate
\[
\|f\|_{\mathrm{BMO}} \lesssim \|\nabla f\|_{L^{d}(\R^d)},
\]
where the space of functions of bounded mean oscillation (BMO) is defined by
\[
\|f\|_{\mathrm{BMO}} := \sup_{Q} \frac{1}{|Q|} \int_{Q} |f(y) - f_{Q}| \, dy,
\]
with \( f_{Q} = \frac{1}{|Q|} \int_{Q} f(y) \, dy \).
In this paper, we establish the inequality
\[
\Bigl\| \frac f{|x|} \Bigr\|_{L^{d,\infty}(\R^d)} \lesssim \|\nabla f\|_{L^{d}(\R^d)},
\]
which provides a different endpoint conclusion. Indeed, membership of a function \( f \in \mathrm{BMO}\) does not guarantee that \( f|x|^{-1} \) belongs to the weak-\(L^d\) space \( L^{d,\infty}(\R^d) \). A classic example confirming this is the function \( f(x) = \log\frac{1}{|x|} \), which is known to be in \( \mathrm{BMO}\). However, a direct computation shows that $|x|^{-1}\log\frac{1}{|x|}\notin L^{d,\infty}(\R^d)$.
\end{remark}

\begin{remark}
By the classical Marcinkiewicz interpolation theorem, when \(\gamma_1 = \gamma_3\), the weak-type norm \(\|\cdot\|_{L^{s,\infty}(\mathbb{R}^d)}\) on the left-hand side of \eqref{WCKN-weak} can be replaced by the strong norm \(\|\cdot\|_{L^{s}(\mathbb{R}^d)}\). Specifically, there exists a constant \(C > 0\) such that
\begin{equation}\label{WCKN-main2}
\bigl\||x|^{\gamma_1} f\bigr\|_{L^{s}(\mathbb{R}^d)}
\le C \bigl\||x|^{\gamma_2}\nabla f\bigr\|_{L^{p}(\mathbb{R}^d)}^{\theta}
\bigl\||x|^{\gamma_1} f\bigr\|_{L^{q,\infty}(\mathbb{R}^d)}^{1-\theta}.
\end{equation}
In the special case where \(\gamma_1 = \gamma_2 = \gamma_3 = 0\), this inequality reduces to
\[
\|f\|_{L^s(\mathbb{R}^d)}
\leq C \|\nabla f\|_{L^{p}(\mathbb{R}^d)}^{\theta} \|f\|_{L^{q,\infty}(\mathbb{R}^d)}^{1-\theta},
\]
which is directly relevant to the analysis of coupled systems arising from the magnetic relaxation theory for stationary Euler flows \cite{MRR2014}.
\end{remark}

\begin{remark}
Our methods also apply to non-homogeneous weights such as $\langle x \rangle^\gamma = (1+|x|^2)^{\gamma/2}$ and anisotropic weights $|x'|^{\beta}|x|^{\gamma}$ with $x=(x_1,\cdots,x_d)=(x',x_d)$, although we focus primarily on power-law weights $|x|^\gamma$ in this paper.
\end{remark}

\subsection{Sketch on the proof}

For $x\in\R^d$ and $0<\alpha<d$, the fractional integral $I_{\alpha}$ of a measurable function $f$ is defined by
$$I_{\alpha}(f)(x): = \int_{\R^d} \frac{f(y)}{|x-y|^{d-\alpha}} \, dy.$$
A key tool connecting gradients to fractional integrals is the following pointwise inequality:
\begin{equation}\label{eq:pointwise-bound}
|f(x)| \lesssim I_1(|\nabla f|)(x), \quad f \in C_c^1(\R^d),
\end{equation}
where \(I_1\) is the fractional integral of order $1$. This pointwise bound is a standard result in harmonic analysis and forms the bridge between gradient estimates and fractional integral estimates in our proof.

Based on this, Our main objective is to establish the inequality
\[
\|\mu^{\frac1q} I_{\alpha}(f)\|_{L^{q,\infty}(\mathbb{R}^d)}\lesssim \|w^{\frac1p}f\|_{L^{p}(\mathbb{R}^d)},
\]
where $1\leq p\leq q<\infty$ and $\frac{\alpha}{d}+\frac{1}{q}-\frac{1}{p}\ge 0$. Combining these estimates with weighted Marcinkiewicz interpolation,
we derive weak-type Caffarelli-Kohn-Nirenberg inequalities under natural parameter conditions.

Our approach combines several techniques from modern harmonic analysis:

\begin{enumerate}
\item \textbf{Sparse domination}:
We say that  $\mathcal{S}:=\{Q_{j,k}\}$ is a sparse family of cubes if:
\begin{enumerate}
\item for each fixed $k$ the cubes $Q_{j,k}$ are pairwise disjoint;
\item if $\Gamma_k=\bigcup_j Q_{j,k}$, then $\Gamma_{k+1}\subset \Gamma_k$;
\item $|\Gamma_{k+1}\bigcap Q_{j,k}|\le \eta |Q_{j,k}|$, $0<\eta<1$.
\end{enumerate}
We use the method of domination by sparse operators, which has revolutionized the study of weighted inequalities \cite{Lerner, LernerOmbrosi}. For a sparse family $\mathcal{S}$ of cubes, we consider operators of the form
\[
\mcal{A}^{\alpha}_{\mcal{S}}f(x) = \sum_{Q\in \mcal{S}} \frac{|Q|^{\frac{\alpha}{d}}}{|Q|} \int_Q |f(y)| dy \cdot \chi_Q(x).
\]
According to \cite[Proposition 3.6,  Proposition 3.9]{Uribe}, $I_{\alpha}$ can be pointwise bounded by such operators. These results allow us to reduce the desired estimate to obtaining two-weight norm inequalities for the sparse operator $\mcal{A}_\S ^\alpha$.

\item \textbf{Two-weight theory}: The project of obtaining weighted versions of important inequalities in harmonic analysis began in the 70s with the work of Muckenhoupt, who characterized the boundedness of the Hardy-Littlewood maximal operator in terms of the so called \(A_p\) condition (\cite{Muckenhoupt}).
     For \(1 < p < \infty\), a weight \(w\) belongs to \(A_p\) if
$$
[w]_{A_p} :=  \sup_{Q}\, \frac{1}{|Q|}\int_{Q}w(y)dy \Big(\frac{1}{|Q|}\int_{Q}w(y)^{-\frac{1}{p-1}}dy\Big)^{p-1}
 < \infty,
$$
where the supremum is taken over cubes $Q$. We say
$$
[w]_{A_1} := \sup_Q \, \frac{1}{|Q|}\int_{Q}w(y)dy \|w^{-1}\chi_Q\|_{L^{\infty}(\R^d)}.
$$
We also define $w\in A_\infty$ if
\[
[w]_{A_\infty}:=\sup_Q \, \frac{1}{|Q|}\int_{Q}w(x)dx\exp\big( \frac{1}{|Q|}\int_{Q}\log w(y)^{-1}dy\big)<\infty.
\]
In \cite{Uribe}, Cruz-Uribe gave a weighted norm bound for sparse operators under the \(A_{p,q}^\alpha\) condition. Let \(1 < p \leq q < \infty\), \(0 \leq \alpha < d\), and \(\mathcal{S}\) a sparse family. If \((\mu, w) \in A_{p,q}^\alpha\) and \(w^{-p'/p}, \mu \in A_\infty\), then
\begin{equation}\label{Strong}
\|\mu^{\frac1q}\mcal{A}^\alpha_\mathcal{S} f\|_{L^q(\R^d)} \lesssim \|w^{\frac1p}f\|_{L^p(\R^d)},
\end{equation}
where
\[[\mu, w]_{A_{p,q}^\alpha} := \sup_Q |Q|^{\frac{\alpha}{d}-1} \|\mu\chi_Q\|_{L^{1}(\R^d)}^{1/q}\left(\frac{1}{|Q|}\int_Q w(y)^{1-p'}dy\right)^{p-1} < \infty.\]

To establish the multiplier weak-type estimates, we will work with pairs of weights $(\mu, w)$ satisfying the following Muckenhoupt-type conditions. For $1 \leq  p \leq q < \infty$ and $0 \leq \alpha < d$, we say $(\mu, w) \in A_{p,q}^{\alpha,*}$ if
\[
[\mu, w]_{A_{p,q}^{\alpha,*}} := \sup_Q |Q|^{\frac{\alpha}{d}-1} \|\mu\chi_Q\|_{L^{1,\infty}(\R^d)}^{1/q} \left(\frac{1}{|Q|}\int_Q w(y)^{1-p'}dy\right)^{p-1} < \infty.
\]

\item \textbf{Multiplier weak-type estimates}: We establish the weak-type boundedness of the fractional integrals in three steps.

(i) \textit{The case $p>1$ with $(\mu,w)\in A_{p,q}^{\alpha,*}$.}
We proceed via sparse domination. The $A_{p,q}^{\alpha,*}$ condition, together with the geometric properties of the sparse family, directly yields the desired estimate.

(ii) \textit{The case $p=1$ with $(\mu,w)\in A_{1,q}^{\alpha}$.} We follow the strategy of Li, Ombrosi and P\'{e}rez \cite{LOP2019} for the Hardy–Littlewood maximal function, and relies on two key tools: the 'pigeon-hole' technique and the Calder\'{o}n–Zygmund decomposition.

(iii) \textit{The case \(p=1\) with \(\mu(x)=|x|^{-d}\) and \(w(x)=|x|^{\alpha-d}\).}
We employ a dyadic decomposition of the space. Applying the conclusion from part (ii), we bound the three terms individually, depending on their separation from the relevant dyadic annulus.

\item \textbf{Weighted Marcinkiewicz interpolation}: We employ weighted versions of the Marcinkiewicz interpolation theorem to handle the interpolation parameter $\theta$.
\end{enumerate}

\subsection{Organization}

The paper is organized as follows. Section \ref{sec:weak-fractional} establishes multiplier weak-type estimates for $\mcal{A}_{\S}^{\alpha}$. Section \ref{sec:necessity} proves the necessity of conditions \eqref{CKN-3}--\eqref{CKN-4} for Theorem \ref{thm:main}. Section \ref{sec:sufficiency} proves the sufficiency, using the tools developed in previous sections. Section \ref{sec:applications} presents some variants of the Caffarelli–Kohn–Nirenberg interpolation inequalities. Finally, Section \ref{sec:appendix} contains technical lemmas about power weights.

\vspace{0.3cm}
Throughout this paper we work in $\mathbb{R}^{d}$ with $d \geq 1$.
We denote by $C_{c}^{1}(\mathbb{R}^{d})$ the space of smooth compactly supported functions,
and by $L^{p,\infty}(\mathbb{R}^{d})$ the weak $L^{p}$ space, equipped with the quasinorm
\[
\|f\|_{L^{p,\infty}(\mathbb{R}^{d})}
    = \sup_{\lambda>0}\; \lambda \, \bigl|\{x\in\mathbb{R}^{d}:|f(x)|>\lambda\}\bigr|^{1/p}.
\]
We use the notation $A \lesssim B$ to mean that $A \leq C B$ for some constant $C>0$
independent of the relevant parameters.
All weight functions $w$ considered in this paper are assumed to be locally integrable
on $\mathbb{R}^{d}$ and to satisfy $0<w(x)<\infty$ for almost every $x\in\mathbb{R}^{d}$.
Finally, for a cube $Q\subset\mathbb{R}^{d}$ we write
\[
\langle f \rangle_{Q} := \frac{1}{|Q|}\int_{Q} f(y)\,dy.
\]

\section{Multiplier weak-type estimates for the sparse operator}\label{sec:weak-fractional}

The main aims of this section are to establish the multiplier weak-type inequalities for $\mcal{A}_{\S}^{\alpha}$.
These inequalities require more subtle proof techniques and differ fundamentally from standard weighted weak-type inequalities, even for the maximal operator. Two principal reasons for this difference are that the weight \(\mu\) acts as a multiplier rather than as a measure, and the weight \(\mu\) is not assumed to satisfy the doubling condition.

Recently, Sweeting \cite{S2025} proved that their necessary condition for the maximal operator is sufficient when $p > 1$.
For \( 1< p < \infty \), we say $w$ is a multiplier $A_p$ weight, written $w\in A_{p}^*:= A^{*}_{p}(\R^d)$ if
\[
[w]_{A_{p}^{*}} :=
\sup_{Q} \; \frac1{|Q|}
\|w\chi_{Q}\|_{L^{1,\infty}(\R^d)}\Big(\frac{1}{|Q|}\int_{Q}w(y)^{1-p'}dy\Big)^{p-1}<\infty.
\]
For \( 1< p,q < \infty \), we say $w$ is a multiplier $A_p$ weight, written $w\in A_{p,q}^*:= A^{*}_{p,q}(\R^d)$ if
\[
[w]_{A_{p,q}^{*}} :=
\sup_{Q} \; \Big(\frac1{|Q|}
\|w^q\chi_{Q}\|_{L^{1,\infty}(\R^d)}\Big)^{\frac 1q}\Big(\frac{1}{|Q|}\int_{Q}w(y)^{-p'}dy\Big)^{\frac1{p'}}<\infty.
\]

For the case $p=1$, it was shown in \cite{MW1977} that \(|x|^{-1}\) is an admissible weight for the maximal operator in dimension \(d=1\) when \(p \geq 1\), and also for the Hilbert transform when \(p = 1\), despite the fact that this function does not belong to any \(A_p\) class. This finding subsequently led to the open problem posed by Muckenhoupt and Wheeden concerning the characterization of weight classes for which such estimates remain valid.

Building upon the ideas in \cite{Uribe}, \cite{MW1977} and \cite{S2025}, we investigate the multiplier weak-type boundedness of the fractional integrals  in the setting of $A_{p,q}^{\alpha,*}$ weights. The analysis will be carried out separately for the cases $p=1$ and $p>1$.

\subsection{The case $p > 1$ with $(\mu, w)\in A_{p,q}^{\alpha,*}$.}

First, we establish weak-type multiplier boundedness for $\mcal{A}_{\S}^{\alpha}$ in the case \( p > 1 \). In contrast to the result of \eqref{Strong}, the condition \( \mu \in A_{\infty} \) is not required here, and we note the strict inclusion \( A_{p,q}^{\alpha,*} \subsetneq A_{p,q}^{\alpha} \).

\begin{thm}\label{thm:Ialpha-p>1}
Let $0 < \alpha < d$ and $1 < p \leq q < \infty$. If $(\mu, w) \in A_{p,q}^{\alpha,*}$ with $w^{1-p'} \in A_\infty$, then
\[
\|\mu^{1/q} \mcal{A}_\S ^\alpha f\|_{L^{q,\infty}(\R^d)} \lesssim \|w^{1/p} f\|_{L^p(\R^d)}.
\]
\end{thm}

\begin{proof}
For each index $k$, set $\Gamma_k=\bigcup_j Q_{j,k}$.  Since $\Gamma_{k+1}\subset \Gamma_k$, we may write
\begin{align*}
\|\mu^{\frac 1q}\mcal{A}_{\mathcal{S}}^{\alpha}(f)\|^q_{L^{q,\infty}(\R^d)}&=
\|\mu \mcal{A}_{\mcal{S}}^{\alpha}(f)^q\|_{L^{1,\infty}(\R^d)}\\
&= \sup_{\lambda>0}\lambda \,
\Bigl|\Bigl\{x\in \mathcal{S}: \mu(x) \mcal{A}_{\mcal{S}}^{\alpha}(f)(x)^q>\lambda\Bigr\}\Bigr|\\
&= \sup_{\lambda>0}\lambda \,
\sum_{k}\Bigl|\Bigl\{x\in \Gamma_k\setminus \Gamma_{k+1} :
\mu(x)  \mcal{A}_{\mcal{S}}^{\alpha}(f)(x)^q >\lambda\Bigr\}\Bigr|\\
&\le \sup_{\lambda>0}\lambda
\sum_{k}\sum_{j}|Q_{j,k}|^{\frac{\alpha q}{d}}\langle f\rangle_{Q_{j,k}}^q\bigl|\bigl\{x\in Q_{j,k} : \mu(x)>\lambda\bigr\}\bigr|.
\end{align*}
Choose numbers $1<r,s<\infty$ satisfying $\frac1r+\frac{1}{p's}=1$.  Then $p>r$ and
\begin{align*}
\|\mu^{\frac 1q}\mcal{A}_{\mathcal{S}}^{\alpha}(f)\|^q_{L^{q,\infty}(\R^d)}
&\le \sum_{k,j}|Q_{j,k}|^{\frac{\alpha q}{d}}
\langle w^{\frac r p} f^{r}\rangle_{Q_{j,k}}^{\frac qr}   \langle w^{(1-p')s}\rangle_{Q_{j,k}}^{\frac q{p's}}
\|\mu\chi_{Q_{j,k}}\|_{L^{1,\infty}(\R^d)}.
\end{align*}
Because $w^{1-p'}\in A_{\infty}$, we can select $s>1$ such that
\[
\Bigl(\frac{1}{|Q|}\int_{Q}w(y)^{(1-p')s}\,dy\Bigr)^{\frac{1}{p's}}
\lesssim \Bigl(\frac{1}{|Q|}\int_{Q}w(y)^{1-p'}\,dy\Bigr)^{\frac{1}{p'}}.
\]
Consequently,
\begin{align*}
\|\mu^{\frac 1q}\mcal{A}_{\mathcal{S}}^{\alpha}(f)\|^q_{L^{q,\infty}(\R^d)}
&\lesssim [\mu, w]^q_{A_{ p,q}^{\alpha,*}}\sum_{k,j} |Q_{j,k}|^{\frac q p}
\langle w^{\frac r {p}} f^{r}\rangle_{Q_{j,k}}^{\frac qr}.
\end{align*}
Using the condition $|Q_{j,k}|\leq 2|E(Q_{j,k})|$ and $q\geq p>r$, we obtain
\begin{align*}
\|\mu^{\frac 1q}\mcal{A}_{\mathcal{S}}^{\alpha}(f)\|^q_{L^{q,\infty}(\R^d)}
&\lesssim
\Bigl(\sum_{k,j} |E(Q_{j,k})|
\langle w^{\frac r {p}} f^{r}\rangle_{Q_{j,k}}^{\frac {p}r}\Bigr)^{\frac q {p}} \\
&\lesssim
\Bigl(\sum_{k,j} \int_{E(Q_{j,k})} M\bigl(|w^{\frac1{p}} f|^{r}\bigr)(x)^{\frac{p}{r}}\,dx
\Bigr)^{\frac{q}{p}} \\
&\lesssim
\bigl\|M\bigl(|w^{\frac1{p}} f|^{r}\bigr)\bigr\|^{\frac{q}{r}}_{L^{\frac{p}r}(\R^d)}
\lesssim \|w^{\frac1{p}} f\|^q_{L^{p}(\R^d)}.
\end{align*}
This completes the proof.
\end{proof}

\subsection{The case $p = 1$ with $(\mu, w)\in A_{1,q}^{\alpha}$}

For the endpoint case $p = 1$, we cannot directly derive the relevant conclusions for a weight pair \((\mu, w)\) belonging to the class \(A_{1,q}^{\alpha,*}\); a stronger condition is required:

\begin{definition}\label{definition:A1qalpha}
Let $1 \leq q < \infty$ and $0 \leq \alpha < d$. A pair $(\mu, w)$ belongs to $A_{1,q}^\alpha$ if
\[
[\mu, w]_{A_{1,q}^\alpha} := \sup_Q |Q|^{\frac{\alpha}{d}-1} \|\mu\chi_Q\|_{L^{1}(\R^d)}^{1/q} \|w^{-1}\chi_Q\|_{L^\infty(\R^d)} < \infty.\]
\end{definition}

We shall need the following reverse H\"older inequality for $A_{\infty}$ weights (see \cite{HyPe}).

\begin{lemma}\label{RHI}
Let $w\in A_{\infty}$ and set $r_{w}=1+\dfrac{1}{\tau_{d}[w]_{A_{\infty}}}$.
Then for any cube $Q$,
\[
\Bigl(\frac{1}{|Q|}\int_{Q}w^{r_{w}}\Bigr)^{1/r_{w}}\le\frac{2}{|Q|}\int_{Q}w.
\]
Consequently, for any measurable set $E\subset Q$,
\[
\frac{w(E)}{w(Q)}\le2\Bigl(\frac{|E|}{|Q|}\Bigr)^{\varepsilon_{w}},\qquad
\varepsilon_{w}=\frac{1}{1+\tau_{d}[w]_{A_{\infty}}}.
\]
\end{lemma}

\begin{thm}\label{thm:Ialpha-p=1}
Let $0 < \alpha < d$ and $1 \leq q < \infty$. If $(\mu, w) \in A_{1,q}^\alpha$ with $\mu \in A_1$ and $w^{1-p'} \in A_\infty$, then
\[
\|\mu^{1/q}\mcal{A}_{\mcal S}^\alpha f\|_{L^{q,\infty}(\R^d)} \lesssim \|w f\|_{L^{1}(\R^d)}.
\]
\end{thm}

\begin{proof}
Set $v:=\frac1{\mu^{1/q}}$. Since $\mu\in A_1$, it follows that $v\in A_{2}\subset A_{\infty}$. It suffices to prove for every bounded, compactly supported function $f$ the inequality
\begin{equation}\label{eq:muv}
\Bigl|\Bigl\{x\in\R^d:1<\frac{\mcal{A}_{\mcal{S}}^{\alpha}(f)(x)}{v(x)}\le2\Bigr\}\Bigr|
\lesssim \|wf\|_{L^{1}(\R^d)}^{q} .
\end{equation}

We decompose the left-hand side of \eqref{eq:muv} as
\[
\sum_{k\in\mathbb{Z}}\Bigl|\Bigl\{x\in\R^d:1<\frac{\mcal{A}_{\mcal{S}}^{\alpha}(f)(x)}{v(x)}\le2,\;a^{k}<v(x)\le a^{k+1}\Bigr\}\Bigr|
:=\sum_{k\in\mathbb{Z}}|E_k|,
\]
where $a>2^{d}$.

For each $k$, define the family $\mathcal{Q}_{k}:=\{I_{j}^{k}\}_{j}$, where $I_{j}^{k}\in\mathcal{S}$ is a maximal cube satisfying
\begin{equation}\label{eq:g}
|I_{j}^{k}|^{\frac{\alpha}{d}}\langle f\rangle_{I_{j}^{k}}  >a^{k}.
\end{equation}
Observe that
\begin{align*}
\{x\in \R^d:\mcal{A}_{\mcal{S}}^{\alpha}(f)(x)>a^{k}\}
&=\bigcup_{\tilde{k}}\Bigl\{x\in\Gamma_{\tilde{k}}\setminus\Gamma_{\tilde{k}+1}:\mcal{A}_{\mcal{S}}^{\alpha}(f)(x)>a^{k}\Bigr\}\\
&=\bigcup_{\tilde{j},\tilde{k}}\bigl\{x\in Q_{\tilde{j},\tilde{k}}:
|Q_{\tilde{j},\tilde{k}}|^{\frac{\alpha}{d}}  \langle f\rangle_{Q_{\tilde{j},\tilde{k}}}>a^{k}\bigr\}
=\bigcup_{j}I_{j}^{k}.
\end{align*}
Hence $E_{k}\subset\Omega_{k}:=\bigcup_{j}I_{j}^{k}$. We now split the collection $\{I_{j}^{k}\}_{j}$ according to the average of $v$:
\[
\mathcal{Q}_{l,k}:=\{I_{j}^{k}\in\mathcal{Q}_{k}:a^{k+l}\le\langle v\rangle_{I_{j}^{k}}<a^{k+l+1}\},\qquad l\ge0,
\]
and
\[
\mathcal{Q}_{-1,k}:=\{I_{j}^{k}\in\mathcal{Q}_{k}:\langle v\rangle_{I_{j}^{k}}<a^{k}\}.
\]

For a cube $I_{j}^{k}\in\mathcal{Q}_{-1,k}$, we perform a Calder\'{o}n--Zygmund decomposition of $v\chi_{I_{j}^{k}}$ at height $a^{k}$, obtaining a family of subcubes $\{I_{j,i}^{k}\}_{i}\subset\mathcal{D}(I_{j}^{k})$ such that
\begin{equation}\label{microlocalCZ}
a^{k}<\langle v\rangle_{I_{j,i}^{k}}<2^{d}a^{k}\quad\text{for all }i,
\end{equation}
and $v(x)\le a^{k}$ for $x\in I_{j}^{k}\setminus\bigcup_{i}I_{j,i}^{k}$.
Define $\Omega_{-1,k}:=\bigcup_{I_{j}^{k}\in\mathcal{Q}_{-1,k}}\bigcup_{i}I_{j,i}^{k}$.

Now write
\begin{align*}
\sum_{k}|E_{k}|
&=\sum_{k}|E_{k}\cap\Omega_{k}|
\leq \sum_{k,j}|E_{k}\cap I_{j}^{k}|  \\
&\le\sum_{k}\sum_{l\ge0}\sum_{I_{j}^{k}\in\mathcal{Q}_{l,k}}a^{q(k+1)}\mu(E_{k}\cap I_{j}^{k})
   +\sum_{k}\sum_{I_{j}^{k}\in\mathcal{Q}_{-1,k}}a^{q(k+1)}\mu(E_{k}\cap I_{j}^{k}) \\
&\le\sum_{k}\sum_{l\ge0}\sum_{I_{j}^{k}\in\Gamma_{l,k}}a^{q(k+1)}\mu(E_{k}\cap I_{j}^{k})
   +\sum_{k}\sum_{i:I_{j,i}^{k}\in\Gamma_{-1,k}}a^{q(k+1)}\mu(I_{j,i}^{k}),
\end{align*}
where
\[
\Gamma_{l,k}:=\{I_{j}^{k}\in\mathcal{Q}_{l,k}:|I_{j}^{k}\cap\{x:a^{k}<v\le a^{k+1}\}|>0\} \; \qquad (l\ge0),
\]
and
\[
\Gamma_{-1,k}:=\{I_{j,i}^{k}\in\mathcal{Q}_{-1,k}:|I_{j,i}^{k}\cap\{x:a^{k}<v\le a^{k+1}\}|>0\}.
\]

We treat the sums corresponding to $l\ge0$ and $l=-1$ separately.
By the monotone convergence theorem, it suffices to obtain a uniform bound for
\[
\sum_{k\ge N}\sum_{l\ge0}\sum_{I_{j}^{k}\in\Gamma_{l,k}}a^{q(k+1)}\mu(E_{k}\cap I_{j}^{k})
+\sum_{k\ge N}\sum_{i:I_{j,i}^{k}\in\Gamma_{-1,k}}a^{q(k+1)}\mu(I_{j,i}^{k}),
\]
where $N<0$.

\subsubsection*{The case $l\ge0$}

Using the Lemma 3.8 in \cite{LOP2019}, there exist constants $c_{1},c_{2}>0$, depending only on $\mu$ and $v$, such that
\[
\mu(E_{k}\cap I_{j}^{k})\le c_{1}e^{-c_{2}l}\,\mu(I_{j}^{k}).
\]
Consequently,
\begin{align*}
&\sum_{k}\sum_{l\ge0}\sum_{I_{j}^{k}\in\Gamma_{l,k}}a^{q(k+1)}\mu(E_{k}\cap I_{j}^{k})\\
&\qquad\le\sum_{l\ge0}c_{1}e^{-c_{2}l}a^{q(1-l)}
        \sum_{k}\sum_{I_{j}^{k}\in\Gamma_{l,k}}\langle v\rangle_{I_{j}^{k}}^{q}\mu(I_{j}^{k})\\
&\qquad=\sum_{l\ge0}c_{1}e^{-c_{2}l}a^{q(1-l)}
        \sum_{k}\sum_{I_{j}^{k}\in\Gamma_{l,k}}|I_{j}^{k}|\,\langle\mu\rangle_{I_{j}^{k}}\langle v\rangle_{I_{j}^{k}}^{q}.
\end{align*}

Fix $l\ge0$ and construct principal cubes for $\bigcup_{k\ge N}\Gamma_{l,k}$.
Let $\mathcal{P}_{0}^{l}$ be the maximal cubes in this union.
Inductively, if $I_{s}^{t}\in\mathcal{P}_{m}^{l}$, then $I_{j}^{k}\in\mathcal{P}_{m+1}^{l}$ if it is a maximal cube (with respect to inclusion) in $\mathcal{D}(I_{s}^{t})$ satisfying
\[
\langle\mu\rangle_{I_{j}^{k}}>2\langle\mu\rangle_{I_{s}^{t}}.
\]
Set $\mathcal{P}^{l}:=\bigcup_{m\ge0}\mathcal{P}_{m}^{l}$ and denote by $\pi(Q)$ the minimal principal cube containing $Q$. Then
\allowdisplaybreaks
\begin{align*}
&\sum_{k}\sum_{l\ge0}\sum_{I_{j}^{k}\in\Gamma_{l,k}}a^{q(k+1)}\mu(E_{k}\cap I_{j}^{k})\\
&\qquad\le\sum_{l\ge0}2c_{1}e^{-c_{2}l}a^{q(1-l)}
        \sum_{I_{s}^{t}\in\mathcal{P}^{l}}\langle\mu\rangle_{I_{s}^{t}}
        \sum_{\substack{k,j\\\pi(I_{j}^{k})=I_{s}^{t}}}|I_{j}^{k}|\,\langle v\rangle_{I_{j}^{k}}^{q}\\
&\qquad\le\sum_{l\ge0}2c_{1}e^{-c_{2}l}a^{q(1-l)}
        \sum_{I_{s}^{t}\in\mathcal{P}^{l}}\langle\mu\rangle_{I_{s}^{t}}
        \Bigl(\sum_{\substack{k,j\\\pi(I_{j}^{k})=I_{s}^{t}}}|I_{j}^{k}|^{-1/q'}v(I_{j}^{k})\Bigr)^{q}\\
&\qquad\lesssim_{d}\sum_{l\ge0}c_{1}e^{-c_{2}l}a^{-ql}[v]_{A_{\infty}^q}
        \sum_{I_{s}^{t}\in\mathcal{P}^{l}}\mu(I_{s}^{t})\langle v\rangle_{I_{s}^{t}}^{q}\\
&\qquad\le a^{q}\sum_{l\ge0}c_{1}e^{-c_{2}l}[v]_{A_{\infty}}^{q}
        \sum_{I_{s}^{t}\in\mathcal{P}^{l}}a^{qt}\mu(I_{s}^{t}).
\end{align*}
Using inequality \eqref{eq:g}, we further bound the last expression by
\begin{align*}
&\le a^{q}\sum_{l\ge0}c_{1}e^{-c_{2}l}[v]_{A_{\infty}}^{q}
        \sum_{I_{s}^{t}\in\mathcal{P}^{l}}|I_{s}^{t}|^{\frac{q\alpha}{d}} \langle f\rangle_{I_{s}^{t}}^{q}  \mu(I_{s}^{t})\\
&\le a^{q}\sum_{l\ge0}c_{1}e^{-c_{2}l}[v]_{A_{\infty}}^{q}
        \Bigl(\sum_{I_{s}^{t}\in\mathcal{P}^{l}}|I_{s}^{t}|^{\frac{\alpha}{d}}\langle f\rangle_{I_{s}^{t}}     \mu(I_{s}^{t})^{1/q}\Bigr)^{q}\\
&\le a^{q} \sum_{l\ge0}c_{1}e^{-c_{2}l}[v]_{A_{\infty}}^{q}
\Bigl(\int_{\mathbb{R}^{d}}f(x)\,
              \sum_{I_{s}^{t}\in\mathcal{P}^{l}}|I_{s}^{t}|^{\frac{\alpha}{d}+\frac1q-1}
              \langle\mu\rangle_{I_{s}^{t}}^{1/q}    \chi_{I_{s}^{t}}(x)\,dx\Bigr)^{q}.
\end{align*}

Fix $x$ and let $\{I_{x}^{m}\}$ be the chain of principal cubes containing $x$.
Since $\langle\mu\rangle_{I_{x}^{m}}$ forms a finite geometric sequence,
\begin{align*}
\sum_{I_{s}^{t}\in\mathcal{P}^{l}}|I_{s}^{t}|^{\frac{\alpha}{d}+\frac1q-1}
\langle\mu\rangle_{I_{s}^{t}}^{\frac1q}     \chi_{I_{s}^{t}}(x)
&=\sum_{0\le m\le m_{0}}|I_{x}^{m}|^{\frac{\alpha}{d}+\frac1q-1}\langle\mu\rangle_{I_{x}^{m}}^{\frac 1q}   \\
&\le\sum_{0\le m\le m_{0}}2^{m-m_{0}}|I_{x}^{m_{0}}|^{\frac{\alpha}{d}+\frac1q-1}
   \langle\mu\rangle_{I_{x}^{m_{0}}}^{\frac 1q} \\
&\le2\,[\mu, w]_{A_{1,q}^{\alpha}}\,w(x).
\end{align*}
Summation over $l$ now yields
\[
\sum_{k}\sum_{l\ge0}\sum_{I_{j}^{k}\in\Gamma_{l,k}}a^{q(k+1)}\mu(E_{k}\cap I_{j}^{k})
\lesssim_{d, q,\mu, w}\,\|wf\|_{L^{1}(\R^d)}^{q}.
\]

\subsubsection*{The case $l=-1$}

Define principal cubes with respect to $\mu$ as follows.
Let $\mathcal{F}_{0}$ be the maximal cubes in $\Gamma_{-1,N}$.
For $m\ge0$, if $I_{s,\ell}^{t}\in\mathcal{F}_{m}$, then $I_{j,i}^{k}\in\mathcal{F}_{m+1}$ if it is a maximal cube in $\mathcal{D}(I_{s,\ell}^{t})$ satisfying
\[
\langle\mu\rangle_{I_{j,i}^{k}}>a^{(k-t)\delta}\langle\mu\rangle_{I_{s,\ell}^{t}},
\]
where $\delta>0$ will be chosen later.
Set $\mathcal{F}:=\bigcup_{m\ge0}\mathcal{F}_{m}$ and again denote by $\pi(Q)$ the minimal principal cube containing $Q$.
If $\pi(I_{j',i'}^{k'})=I_{s,\ell}^{t}$, then $k'\ge t$ and by construction
\[
\langle\mu\rangle_{I_{j',i'}^{k'}}\le a^{(k'-t)\delta}\langle\mu\rangle_{I_{s,\ell}^{t}}.
\]

Now,
\begin{align*}
\sum_{k}\sum_{I_{j,i}^{k}\in\Gamma_{-1,k}}a^{q(k+1)}\mu(I_{j,i}^{k})
&\le a^{q}\sum_{k}\sum_{I_{j,i}^{k}\in\Gamma_{-1,k}}\langle v\rangle_{I_{j,i}^{k}}^{q}\mu(I_{j,i}^{k})\\
&\le a^{q}\sum_{I_{s,\ell}^{t}\in\mathcal{F}}\langle\mu\rangle_{I_{s,\ell}^{t}}
   \sum_{\substack{k,j,i\\ \pi(I_{j,i}^{k})=I_{s,\ell}^{t}}}a^{q(k-t)\delta}
   |I_{j,i}^{k}|^{1-q}v(I_{j,i}^{k})^{q}\\
&\le a^{q}\sum_{I_{s,\ell}^{t}\in\mathcal{F}}\langle\mu\rangle_{I_{s,\ell}^{t}}
   \sum_{k\ge t}a^{q(k-t)\delta}
   \sum_{\substack{j,i\\ \pi(I_{j,i}^{k})=I_{s,\ell}^{t}}}|I_{j,i}^{k}|^{1-q}v(I_{j,i}^{k})^{q}.
\end{align*}
Following the same reasoning as in \cite[Lemma 2.1]{LOP2019}, the family $\Gamma:=\bigcup_{l\ge-1}\bigcup_{k\ge N}\Gamma_{l,k}$ is sparse and satisfies
\begin{equation}\label{sparse}
\Bigl|\bigcup_{\substack{Q',Q\in\Gamma\\ Q'\subsetneq Q}}Q'\Bigr|
\le\frac{2^{d}}{a}\,|Q|.
\end{equation}
By the sparsity condition \eqref{sparse},
\[
\sum_{\substack{j,i\\ \pi(I_{j,i}^{k})=I_{s,\ell}^{t}}}|I_{j,i}^{k}|
\le\Bigl(\frac{2^{d}}{a}\Bigr)^{k-t}|I_{s,\ell}^{t}|,
\]
and Lemma \ref{RHI} gives
\[
\sum_{\substack{j,i\\ \pi(I_{j,i}^{k})=I_{s,\ell}^{t}}}v(I_{j,i}^{k})
\le2\Bigl(\frac{2^{d}}{a}\Bigr)^{\frac{k-t}{2\tau_{d}[v]_{A_{\infty}}}}v(I_{s,\ell}^{t}).
\]
Choose $\delta:=1/(c_{d}'[v]_{A_{\infty}})$ with $c_{d}'$ sufficiently large (depending only on $d$). Then
\begin{align*}
&\sum_{k\ge t}a^{q(k-t)\delta}
   \sum_{\substack{j,i\\ \pi(I_{j,i}^{k})=I_{s,\ell}^{t}}}|I_{j,i}^{k}|^{1-q}v(I_{j,i}^{k})^{q}\\
&\qquad\le\sum_{k\ge t}a^{q(k-t)\delta}
        \Bigl(\sum_{\substack{j,i\\ \pi(I_{j,i}^{k})=I_{s,\ell}^{t}}}|I_{j,i}^{k}|^{1/q-1}v(I_{j,i}^{k})\Bigr)^{q}\\
&\qquad\le c_{d}[v]_{A_{\infty}}^{q}\,|I_{s,\ell}^{t}|^{1-q}v(I_{s,\ell}^{t})^{q}.
\end{align*}
It remains to bound
\begin{align*}
&\sum_{I_{s,\ell}^{t}\in\mathcal{F}}\langle\mu\rangle_{I_{s,\ell}^{t}}|I_{s,\ell}^{t}|^{1-q}v(I_{s,\ell}^{t})^{q}
 =\sum_{I_{s,\ell}^{t}\in\mathcal{F}}\langle v\rangle_{I_{s,\ell}^{t}}^{q}\mu(I_{s,\ell}^{t})\\
&\qquad\le\sum_{I_{s,\ell}^{t}\in\mathcal{F}}a^{q(t+1)}\mu(I_{s,\ell}^{t})
\le a^{q}\sum_{I_{s,\ell}^{t}\in\mathcal{F}}|I_{s}^{t}|^{\frac{q\alpha}{d}}
           \langle f\rangle_{I_{s}^{t}}^{q}       \mu(I_{s,\ell}^{t})\\
&\qquad\leq a^{q} \int_{\mathbb{R}^{d}}f(x)\,
           \Bigl(\sum_{I_{s,\ell}^{t}\in\mathcal{F}}|I_{s}^{t}|^{\frac{\alpha}{d}-1}
           \mu(I_{s,\ell}^{t})^{\frac 1q}\chi_{I_{s}^{t}}(x)\Bigr)^{q}dx.
\end{align*}
Thus, we need to show that
\begin{equation}\label{eq:goal}
\sum_{I_{s,\ell}^{t}\in\mathcal{F}}|I_{s}^{t}|^{\frac{\alpha}{d}-1}
\mu(I_{s,\ell}^{t})^{1/q} \chi_{I_{s}^{t}}(x)
\lesssim_{d,\mu,w} w(x).
\end{equation}

Fix $x$. For each $t$, there is at most one cube $I^{t}$ containing $x$ with $\langle v\rangle_{I^{t}}\le a^{t}$; let $G:=\{I^{t}:I^{t}\ni x\}$.
Construct principal cubes for $G$: take $\mathcal{G}_{0}=\{I^{k_{0}}\}$ as the maximal cube in $G$, and if $I^{k_{m}}\in\mathcal{G}_{m}$, let $I^{k_{m+1}}\in\mathcal{G}_{m+1}$ be a maximal subcube satisfying $\langle\mu\rangle_{I^{k_{m+1}}}>2\langle\mu\rangle_{I^{k_{m}}}$.
Then
\begin{align*}
&\sum_{I_{s,\ell}^{t}\in\mathcal{F}}|I_{s}^{t}|^{\frac{\alpha}{d}-1}
   \mu(I_{s,\ell}^{t})^{1/q}   \chi_{I_{s}^{t}}(x) \\
&=\sum_{I_{s,\ell}^{t}\in\mathcal{F}}|I^{t}|^{\frac{\alpha}{d}-1}
        \mu(I_{s,\ell}^{t})^{1/q}
=\sum_{I_{s,\ell}^{t}\in\mathcal{F}}|I^{t}|^{\frac{\alpha}{d}}
        \langle\mu\rangle_{I^{t}}^{1/q}
        \Bigl(\frac{\mu(I_{s,\ell}^{t})}{\mu(I^{t})}\Bigr)^{1/q}\\
&\le\sum_{m}|I^{k_{m}}|^{\frac{\alpha}{d}+\frac1q-1}\mu(I^{k_{m}})^{1/q}
        \sum_{\substack{I^{t}\in G\\ k_{m}\le t<k_{m+1}}}
        \sum_{s,\ell:I_{s,\ell}^{t}\in\mathcal{F}}
        \Bigl(\frac{\mu(I_{s,\ell}^{t})}{\mu(I^{t})}\Bigr)^{1/q}.
\end{align*}
If we can prove
\begin{equation}\label{eq:goal1}
\sum_{\substack{I^{t}\in G\\ k_{m}\le t<k_{m+1}}}
\sum_{s,\ell:I_{s,\ell}^{t}\in\mathcal{F}}
\Bigl(\frac{\mu(I_{s,\ell}^{t})}{\mu(I^{t})}\Bigr)^{1/q}
\le c_{d,\mu},
\end{equation}
then \eqref{eq:goal} follows immediately.

To establish \eqref{eq:goal1}, denote by $I_{j,i}^{k_{m}}\subset I^{k_{m}}$ the cube that contains $I_{s,\ell}^{t}$ (such a cube exists by the construction of $\Gamma$).
Let $I_{j',i'}^{k'}=\pi(I_{j,i}^{k_{m}})\in\mathcal{F}$.
We have
\[
\langle\mu\rangle_{I_{s,\ell}^{t}}>a^{(t-k')\delta}\langle\mu\rangle_{I_{j',i'}^{k'}},\qquad
\langle\mu\rangle_{I_{j,i}^{k_{m}}}\le a^{(k_{m}-k')\delta}\langle\mu\rangle_{I_{j',i'}^{k'}}.
\]
Consequently,
\begin{align*}
\langle\mu\rangle_{I_{s,\ell}^{t}}
&>a^{(t-k_{m})\delta}\langle\mu\rangle_{I_{j,i}^{k_{m}}}\ge a^{(t-k_{m})\delta}\,\operatorname*{ess\,inf}_{y\in I_{j,i}^{k_{m}}}\mu(y)\\
&\ge\frac{a^{(t-k_{m})\delta}}{[\mu]_{A_{1}}}\langle\mu\rangle_{I^{k_{m}}}\ge\frac{a^{(t-k_{m})\delta}}{2[\mu]_{A_{1}}}\langle\mu\rangle_{I^{t}}.
\end{align*}
Hence, for almost every $y\in I_{s,\ell}^{t}$,
\[
\mu(y)\ge\frac{a^{(t-k_{m})\delta}}{2[\mu]_{A_{1}}^{2}}\langle\mu\rangle_{I^{t}}=:\lambda.
\]
Therefore,
\begin{align*}
\sum_{s,\ell:I_{s,\ell}^{t}\in\mathcal{F}}\mu(I_{s,\ell}^{t})^{1/q}
&\le\mu\bigl(\{y\in I^{t}:\mu(y)>\lambda\}\bigr)^{1/q}\\
&\le2\,\mu(I^{t})^{1/q}
   \Bigl(\frac{|\{y\in I^{t}:\mu(y)>\lambda\}|}{|I^{t}|}\Bigr)^{\frac1{2q\tau_{d}[\mu]_{A_{1}}}}\\
&\le2^{1/q}\mu(I^{t})^{1/q}
   \Bigl(\lambda^{-1}\langle\mu\rangle_{I^{t}}\Bigr)^{\frac1{c_{d,q}[\mu]_{A_{1}}}}\\
&\lesssim_{d}\mu(I^{t})^{1/q}
   a^{\frac{t-k_{m}}{c_{d}[\mu]_{A_{1}}[v]_{A_{\infty}}}}.
\end{align*}
Summation over $t$ now gives \eqref{eq:goal1}, which completes the proof of \eqref{eq:goal} and hence of the theorem.
\end{proof}

\section{Necessity of the conditions}\label{sec:necessity}

In this section, we prove that conditions \eqref{CKN-3}--\eqref{CKN-4} are necessary for Theorem \ref{thm:main}.

\begin{proof}[Proof of necessity]
Let $f \in C_c^1(\R^d)$ and consider the scaling $f_\lambda(x) = f(\lambda x)$ for $\lambda > 0$. Then
\[
\||x|^{\gamma_1} f_\lambda\|_{L^{s,\infty}(\R^d)} = \lambda^{-\gamma_1 - d/s} \||x|^{\gamma_1} f\|_{L^{s,\infty}(\R^d)},
\]
\[
\||x|^{\gamma_2} \nabla f_\lambda\|_{L^p(\R^d)} = \lambda^{1 - \gamma_2 - d/p} \||x|^{\gamma_2} \nabla f\|_{L^p(\R^d)},
\]
\[
\||x|^{\gamma_3} f_\lambda\|_{L^{q,\infty}(\R^d)} = \lambda^{-\gamma_3 - d/q} \||x|^{\gamma_3} f\|_{L^{q,\infty}(\R^d)}.
\]
Substituting into \eqref{WCKN-weak} gives
\[
\lambda^{-\gamma_1 - d/s} \lesssim \lambda^{\theta(1 - \gamma_2 - d/p) - (1-\theta)(\gamma_3 + d/q)}.
\]
Taking $\lambda \to 0^+$ and $\lambda \to \infty$ yields condition \eqref{CKN-3}.

Fixing a $v\in C_c^{\infty}(B_1(0))\setminus\{0\}$,
   we consider $u(x):=v(x-x_0)$ with $|x_0|=R$ large. Then $u\in C^{\infty}_c(B_1(x_0))$, $|x| \sim R$ on the support, and \eqref{WCKN-weak} implies
\[
R^{\gamma_1} \lesssim R^{\theta\gamma_2 + (1-\theta)\gamma_3},
\]
which gives \eqref{CKN-4} as $R \to \infty$.
\end{proof}

\section{Sufficiency of the conditions}\label{sec:sufficiency}

We now prove the sufficiency part of Theorem \ref{thm:main}. The strategy is as follows:
\begin{enumerate}
\item Use the pointwise inequality \eqref{eq:pointwise-bound} to bound $|f|$ by $I_1(|\nabla f|)$.
\item Apply weak-type estimates for fractional integrals from Section \ref{sec:weak-fractional}.
\item Use Marcinkiewicz interpolation to handle the interpolation parameter $\theta$.
\end{enumerate}

\subsection{Multiplier weak-type estimates for fractional integrals with power weights}

Recall that the standard dyadic grid in $\mathbb{R}^d$ consists of the cubes
\[
  [0,2^{-k})^d+2^{-k}j,\qquad k\in\mathbb{Z}, j\in\mathbb{Z}^d.
\]
Denote the standard dyadic grid by $\mathscr{D}$.

By a general dyadic grid $\mathscr{D}$ we mean a collection of cubes with the following
properties: (i) for any $Q\in\mathscr{D}$ its sidelength $l_Q$ is of the form $2^k$, $k\in\mathbb{Z}$; (ii)
 $Q\cap R \in \{Q,R,\varnothing\}$ for any $Q,R\in\mathscr{D}$; (iii) the cubes of a fixed sidelength $2^k$ form a partition
of $\mathbb{R}^d$.

According to \cite{Uribe}, $I_{\alpha}$ can be pointwise bounded by such operators $\mathcal{A}_{\mathcal{S}}^{\alpha}$. Specifically, the following two propositions hold.
\begin{prop}[\cite{Uribe}, Proposition 3.6]\label{prop2.1}
There exist dyadic lattices $\mcal{D}^1,\dots,\mcal{D}^{3^d}$ such that for any nonnegative, bounded, compactly supported measurable function $f$,
\[
\mcal{A}^\alpha_{\D^j} f(x)\lesssim_{d,\alpha} I_\alpha (f)(x)\lesssim_{d,\alpha} \sum_{j=1}^{3^d}\mcal{A}^\alpha_{\D^j}f(x),\qquad \forall x\in \R^d.
\]
\end{prop}
\begin{prop}[\cite{Uribe}, Proposition 3.9]\label{prop2.2}
Let $\D$ be a dyadic lattice, $\alpha\in (0,d)$, and let $f\ge 0$ be a bounded, compactly supported measurable function. Then one can find a sparse family $\mcal{S}=\mcal{S}(f)\subseteq \D$ such that
\[
\mcal{A}^\alpha_\mcal{S}f(x)\leq \mcal{A}^\alpha _\D f(x)\lesssim_{d,\alpha} \mcal{A}^\alpha_\mcal{S}f(x),\qquad \forall x\in \R^d.
\]
\end{prop}
These results allow us to reduce the desired estimate to obtaining two-weight norm inequalities for the sparse operator $\mcal{A}_\S ^\alpha$.

\begin{thm}\label{thm:Ialpha-power}
    Let \(0<\alpha<d\), \(1\le p\leq q<\infty\), and suppose that
    \[
    \alpha-\frac{d}{p}+\frac{d}{q}=\frac{b}{p}-\frac{a}{q}\ge 0,\qquad a\ge -d .
    \]
and
        \[
        b < d(p-1)\;\; \text{for } p > 1, \qquad
        b \le 0 \;\; \text{for } p = 1.
        \]
    Then
    \[
    \bigl\||x|^{\frac{a}{q}} I_{\alpha}f\bigr\|_{L^{q,\infty}(\R^d)}
        \lesssim \bigl\||x|^{\frac{b}{p}}f\bigr\|_{L^{p}(\R^d)}.
    \]
\end{thm}
\begin{proof}
Set \( f_N = \min\{f,N\}\chi_{B(0,N)}\). By Propositions \ref{prop2.1} and \ref{prop2.2} we have
\[
I_{\alpha} (f_N)(x)\lesssim_{d,\alpha}\sum_{j=1}^{3^{d}}\mathcal{A}_{\mathcal{S}^{j}}^{\alpha}(f_N)(x),
\]
where each \( \mathcal{S}^{j}\subseteq \mathcal{D}^{j} \) is a sparse family. Consequently,
\begin{align*}
\bigl\|\mu^{\frac 1q}  I_{\alpha} (f_N)\bigr\|_{L^{q,\infty}(\R^d)}
&\lesssim\sum_{j=1}^{3^{d}}\bigl\|\mu^{\frac 1q} \mcal{A}_{\mathcal{S}^{j}}^{\alpha}(f_N)\bigr\|_{L^{q,\infty}(\R^d)} \\
&\lesssim\sum_{j=1}^{3^{d}}\bigl\|w^{\frac1{p}} f_N\bigr\|_{L^{p}(\R^d)}\lesssim\bigl\| w^{\frac1{p}} f\bigr\|_{L^{p}(\R^d)}.
\end{align*}
Finally, using the weak \(L^{q,\infty}(\R^d)\) lower-semicontinuity of the quasinorm, we obtain
\begin{align*}
\bigl\|\mu^{\frac 1q} I_{\alpha} (f)\bigr\|_{L^{q,\infty}(\R^d)}
&\leq \liminf_{N\rightarrow +\infty}\bigl\|\mu^{\frac 1q} I_{\alpha}(f_N)\bigr\|_{L^{q,\infty}(\R^d)}
\lesssim \bigl\| w^{\frac1{p}} f\bigr\|_{L^{p}(\R^d)}.
\end{align*}
Combining Theorems \ref{thm:Ialpha-p>1} and \ref{thm:Ialpha-p=1} with Lemma \ref{lem:power-1}, we conclude that Theorem \ref{thm:Ialpha-power} holds in all cases except when $p=1$, $\mu(x)=|x|^{-d}$, and $w(x)=|x|^{\alpha-d}$.

We now proceed to establish the result for this remaining case.
By the definition of the weak Lebesgue space, it suffices to show that for any nonnegative function $f$,
\begin{equation}\label{S3-3-eq0}
\bigl|\{x\in \R^d: |x|^{-\frac dq}I_{\alpha}(f)(x)>\lambda\}\bigr|
\lesssim \lambda^{-q}\bigl\||x|^{\alpha-d}f\bigr\|_{L^{1}(\R^d)}.
\end{equation}
For each integer $k$, define
\begin{align*}
G_k=\{2^k\le |x|<2^{k+1}\},\qquad
H_k=\{ |x|\le 2^{k-1}\},\\
I_k=\{2^{k-1}<|x|\le 2^{k+2}\},\qquad
J_k=\{ |x|>2^{k+2}\}.
\end{align*}
Given $\lambda>0$, we decompose the set in \eqref{S3-3-eq0} into three parts:
\begin{align*}
A_1&:=\sum_{k=-\infty}^{\infty}\Bigl|G_k\cap \bigl\{|x|^{-\frac dq}I_{\alpha}(f\chi_{H_k})(x)>\lambda/3\bigr\}\Bigr|,\\
A_2&:=\sum_{k=-\infty}^{\infty}\Bigl|G_k\cap \bigl\{|x|^{-\frac dq}I_{\alpha}(f\chi_{I_k})(x)>\lambda/3\bigr\}\Bigr|,\\
A_3&:=\sum_{k=-\infty}^{\infty}\Bigl|G_k\cap \bigl\{|x|^{-\frac dq}I_{\alpha}(f\chi_{J_k})(x)>\lambda/3\bigr\}\Bigr|.
\end{align*}
It then remains to bound $A_1$, $A_2$, and $A_3$ by the right-hand side of \eqref{S3-3-eq0}.

For $x\in G_k$, note that $|x-y|\ge |x|/2$ when $y\in H_k$. Hence,
\[
I_{\alpha}(f\chi_{H_k})(x)
\le 2^{\delta_1}|x|^{\delta_1}\int_{|y|\le |x|}\frac{f(y)}{|x-y|^{d-\alpha+\delta_1}}\,dy
\le 2^{\delta_1}|x|^{\delta_1}I_{\alpha-\delta_1}(f)(x),
\]
where $0<\delta_1<\min\{\alpha,d/q\}$. Consequently, the set appearing in $A_1$ is contained in
\[
\bigl\{|x|^{-\frac dq+\delta_1}I_{\alpha-\delta_1}(f)(x)>\lambda/(3\cdot2^{\delta_1})\bigr\}.
\]
and
\[
A_1\le\Bigl(\frac{3\cdot2^{\delta_1}}{\lambda}\Bigr)^{q}
\bigl\||x|^{-\frac dq+\delta_1}I_{\alpha-\delta_1}(f)\bigr\|_{L^{q,\infty}(\R^d)}^{q}.
\]
By Theorem \ref{thm:Ialpha-p=1} and the fact that $(|x|^{-d+q\delta_1},|x|^{\alpha-d})\in A_{1,q}^{\alpha-\delta_1}$, the right-hand side is bounded by $C\lambda^{-q}\bigl\||x|^{\alpha-d}f\bigr\|_{L^{1}(\R^d)}^{q}$, which is the desired bound.

Take $0<\delta_2<d/q$. For $x\in G_k$, we have $|x|\ge2^k$. Thus,
\[
|x|^{-\frac dq}I_{\alpha}(f\chi_{I_k})(x)
\le 2^{-k\delta_2}|x|^{-\frac dq+\delta_2}I_{\alpha}(f\chi_{I_k})(x)
\le 2^{-(k-1)\delta_2}|x|^{-\frac dq+\delta_2}I_{\alpha}(f\chi_{I_k})(x).
\]
Therefore, the set defining $A_2$ is contained in
\[
\bigcup_{k}\bigl\{x\in\R^d:|x|^{-\frac dq+\delta_2}I_{\alpha}(f\chi_{I_k})(x)
>\lambda\cdot2^{(k-1)\delta_2}/3\bigr\}.
\]
Applying Theorem \ref{thm:Ialpha-p=1} with the pair $(|x|^{-d+q\delta_2},|x|^{\alpha-d-\delta_2})\in A_{1,q}^{\alpha}$ gives
\[
A_2\le\sum_{k}\Bigl(\frac{3}{\lambda\cdot2^{(k-1)\delta_2}}
\int_{I_k}f(y)|y|^{\alpha-d-\delta_2}\,dy\Bigr)^{q}.
\]
Since $|y|\le2^{k+2}\le8\cdot2^{k-1}$ for $y\in I_k$, we have $2^{(k-1)\delta_2}\ge(|y|/8)^{\delta_2}$. Hence,
\[
A_2\le C\lambda^{-q}\sum_{k}\Big(\int_{I_k}f(y)|y|^{\alpha-d}\,dy\Big)^q
\le C\lambda^{-q}\Big(\sum_{k}\int_{I_k}f(y)|y|^{\alpha-d}\,dy\Big)^q
\leq C\lambda^{-q}\bigl\||x|^{\alpha-d}f\bigr\|_{L^{1}(\R^d)}^{q},
\]
since each point belongs to at most three sets $I_k$.

For $x\in G_k$ and $y\in J_k$, we have $|y|\ge2|x|$ and $|x-y|\ge|y|/2$. Consequently,
\[
I_{\alpha}(f\chi_{J_k})(x)
\le2^{d-\alpha}\int_{J_k}\frac{f(y)}{|y|^{d-\alpha}}\,dy
\le2^{d-\alpha}\int_{\R^d}f(y)|y|^{\alpha-d}\,dy.
\]
Thus,
\[
A_3\le\sum_{k}\Bigl|G_k\cap\bigl\{2^{d-\alpha}|x|^{-\frac dq}\big\||y|^{\alpha-d}f\big\|_{L^{1}(\R^d)}>\lambda/3\bigr\}\Bigr|.
\]
Since the sets $G_k$ are disjoint and cover $\R^d$, this equals
\[
\Bigl|\bigl\{x\in\R^d:2^{d-\alpha}|x|^{-\frac dq}\big\||y|^{\alpha-d}f\big\|_{L^{1}(\R^d)}>\lambda/3\bigr\}\Bigr|.
\]
This set is a ball of radius $\bigl(3\cdot2^{d-\alpha}\lambda^{-1}\big\||y|^{\alpha-d}f\big\|_{L^{1}(\R^d)}\bigr)^{q/d}$, then
$$
A_3\le C\lambda^{-q}\bigl\||x|^{\alpha-d}f\bigr\|_{L^{1}(\R^d)}^{q}.
$$

Combining the estimates for $A_1$, $A_2$ and $A_3$ completes the proof of Theorem \ref{thm:Ialpha-power}.
\end{proof}

\subsection{Weighted Marcinkiewicz interpolation}

To handle weighted estimates in our setting, we require a weighted extension of the classical Marcinkiewicz interpolation theorem. The statement is as follows.

\begin{thm}[Weighted Marcinkiewicz interpolation]\label{thm:weighted-marcinkiewicz}
Let $0 < p_0, p_1 < \infty$ and $0 < v_0, v_1 < \infty$ be measurable functions on $\R^d$. For $0 < \theta < 1$, define
\[
\frac{1}{p_\theta} = \frac{\theta}{p_0} + \frac{1-\theta}{p_1},  \quad v_\theta = v_0^{\theta} v_1^{1-\theta}.
\]
If $f v_i \in L^{p_i, \infty}(\R^d)$ for $i = 0, 1$, then $f v_\theta \in L^{p_\theta, \infty}(\R^d)$ with
\[
\| f v_\theta \|_{L^{p_\theta,\infty}(\R^d)} \le 2\| f v_0 \|_{L^{p_0, \infty}(\R^d)}^{\theta}\| f v_1 \|_{L^{p_1, \infty}(\R^d)}^{1-\theta}.
\]
\end{thm}

\begin{proof}
Since $f v_i \in L^{p_i, \infty}(\R^d)$ for $i=0,1$, we have
\[
M_i:=\| f v_i \|_{L^{p_i, \infty}} = \sup_{\lambda > 0} \lambda \cdot \bigl( m_{f v_i}(\lambda) \bigr)^{1/p_i} < \infty,
\]
where $m_g(\lambda)=|\{ x \in \R^d : |g(x)| > \lambda \}|$ denotes the distribution function.

Fix $\lambda > 0$. We decompose the set where $|f(x) v_\theta(x)| > \lambda$ as
\begin{align*}
&\bigl\{ x\in \R^d : |f(x)| v_0(x)^{\theta} v_1(x)^{1-\theta}> \lambda \bigr\}\\
&\subset \bigl\{ x\in \R^d : |f(x)| v_0(x) > \epsilon^{\theta-1} \lambda^\alpha \bigr\}
\bigcup \bigl\{ x\in \R^d : |f(x)| v_1(x) > \epsilon^{\theta}\lambda^\beta \bigr\},
\end{align*}
with parameters $\alpha,\beta,\epsilon>0$ to be chosen later. Define
\[
\alpha = \frac{p_\theta}{p_0},\qquad \beta = \frac{p_\theta}{p_1}.
\]
Then
\[
\alpha\theta+\beta(1-\theta)
= \frac{p_\theta}{p_0}\theta+\frac{p_\theta}{p_1}(1-\theta)
= p_\theta\left(\frac{\theta}{p_0}+\frac{1-\theta}{p_1}\right)
= p_\theta\cdot\frac{1}{p_\theta}=1.
\]
With this choice of $\alpha$ and $\beta$ we obtain
\[
m_{f v_\theta}(\lambda)
\le m_{f v_0}\bigl(\epsilon^{\theta-1}\lambda^\alpha\bigr) +m_{f v_1}\bigl(\epsilon^{\theta}\lambda^\beta\bigr).
\]
Using the weak $L^{p_i}$ bounds,
\begin{align*}
m_{f v_0}\bigl(\epsilon^{\theta-1}\lambda^\alpha\bigr)
&\le \frac{M_0^{p_0}}{\epsilon^{(\theta-1) p_0}\lambda^{\alpha p_0}}
= \frac{M_0^{p_0}}{\epsilon^{(\theta-1) p_0}\lambda^{p_\theta}},\\
m_{f v_1}\bigl(\epsilon^{\theta}\lambda^\beta\bigr)
&\le \frac{M_1^{p_1}}{\epsilon^{\theta p_1}\lambda^{\beta p_1}}
= \frac{M_1^{p_1}}{\epsilon^{\theta p_1}\lambda^{p_\theta}},
\end{align*}
because $\alpha p_0=p_\theta$ and $\beta p_1=p_\theta$. Choose $\epsilon>0$ by
\[
\epsilon=\Bigl(\frac{M_1^{p_1}}{M_0^{p_0}}\Bigr)^{\frac{p_\theta}{p_0p_1}}.
\]
Then
\[
\frac{M_0^{p_0}}{\epsilon^{(\theta-1) p_0}}
= \frac{M_1^{p_1}}{\epsilon^{\theta p_1}}
= M_0^{\theta p_\theta}M_1^{(1-\theta) p_\theta}.
\]
Consequently,
\[
m_{f v_\theta}(\lambda)\le 2M_0^{\theta p_\theta}M_1^{(1-\theta) p_\theta},
\]
and therefore
\[
\lambda\,\bigl(m_{f v_\theta}(\lambda)\bigr)^{1/p_\theta}
\le 2M_0^{\theta}M_1^{1-\theta}.
\]
Taking the supremum over all $\lambda>0$ yields
\[
\| f v_\theta \|_{L^{p_\theta,\infty(\R^d)}}
\le 2M_0^{\theta}M_1^{1-\theta}<\infty.
\]
Hence $f v_\theta\in L^{p_\theta,\infty}(\R^d)$.
\end{proof}
\begin{remark}\label{Marcinkiewicz}
When \(v_0 = v_1\), the inequality in Theorem \ref{thm:weighted-marcinkiewicz} recovers the classical Marcinkiewicz interpolation estimate in the strong Lebesgue spaces:
\[
\| f v_\theta \|_{L^{p_\theta}(\R^d)}
\le 2\,\| f v_0 \|_{L^{p_0,\infty}(\R^d)}^{\,\theta}
      \| f v_1 \|_{L^{p_1,\infty}(\R^d)}^{\,1-\theta}.
\]
\end{remark}

\subsection{Proof of the sufficiency part of Theorem \ref{thm:main}}

When \(\theta = 0\), it suffices to apply condition \eqref{CKN-3} and the weak Lebesgue space version of H\"{o}lder’s inequality.
When \(\theta = 1\), we only need conditions \eqref{CKN-3} and together with Theorem \ref{thm:Ialpha-power}.

We now consider the case \(0 < \theta < 1\). Take \(p_0 \ge p\) and \(\gamma_0 \le \gamma_2\) such that
\[
\frac{1}{s}= \frac{\theta}{p_0}+\frac{1-\theta}{q},\qquad
\gamma_1 = \theta\gamma_0+(1-\theta)\gamma_3 .
\]
Such numbers \(p_0\) and \(\gamma_0\) exist because of conditions \eqref{CKN-3} and \eqref{CKN-4}.
Moreover, condition \eqref{CKN-3} implies the identity
\begin{equation}\label{main-eq1-1}
1-\frac{d}{p}+\frac{d}{p_0}= \gamma_2-\gamma_0,
\end{equation}
condition \eqref{WCKN-2-0} implies that
\begin{equation}\label{main-eq1-2}
\gamma_0\geq -\frac d{p_0}.
\end{equation}
With the conditions \eqref{main-eq1-1} and \eqref{main-eq1-2}, the pointwise estimate \eqref{eq:pointwise-bound} and Theorem \ref{thm:Ialpha-power} together yield
\begin{equation}\label{main-eq2}
\bigl\||x|^{\gamma_0} f\bigr\|_{L^{p_0,\infty}(\R^d)}
\lesssim \bigl\||x|^{\gamma_0} I_1(|\nabla f|)\bigr\|_{L^{p_0,\infty}(\R^d)}
\lesssim \bigl\||x|^{\gamma_2}\nabla f\bigr\|_{L^{p}(\R^d)}.
\end{equation}
Applying Theorem \ref{thm:weighted-marcinkiewicz} with exponent \(\theta\) gives
\begin{equation}\label{main-eq3}
\bigl\||x|^{\gamma_1}f\bigr\|_{L^{s,\infty}(\R^d)}
\lesssim \bigl\||x|^{\gamma_0}f\bigr\|_{L^{p_0,\infty}(\R^d)}^{\theta}
        \bigl\||x|^{\gamma_3}f\bigr\|_{L^{q,\infty}(\R^d)}^{1-\theta}.
\end{equation}
Finally, combining \eqref{main-eq2} and \eqref{main-eq3} yields the required inequality.
\qed

\section{Further results}\label{sec:applications}

This section explores different versions of the weak type Caffarelli–Kohn–Nirenberg interpolation inequalities.

\begin{thm}\label{thm:main2}
Suppose $s, p, q, \gamma_1, \gamma_2, \gamma_3, \theta$ satisfy $s, p, q > 1$,
\begin{equation*}\label{WCKNr-2}
   \frac{1}{p} + \frac{\gamma_2}{d} < 1, \qquad  \frac{1}{p} + \frac{\gamma_2-1}{d}>0, \qquad
    \frac{1}{s} + \frac{\gamma_1}{d} = \frac12\Big(\frac{1}{p} + \frac{\gamma_2 - 1}{d}\Big) + \frac12\Big(\frac{1}{q} + \frac{\gamma_3}{d}\Big),
\end{equation*}
and
\begin{equation*}\label{WCKNr-4}
\gamma_1 \le \frac {\gamma_2}2 + \frac{\gamma_3}2.
\end{equation*}
Then there exists a constant $C> 0$ such that
\begin{equation*}\label{WCKNr-weak}
\big\|\,|x|^{\gamma_1} f\,\big\|_{L^{s}(\R^d)}
\le C\,\big\|\,|x|^{\gamma_2}\nabla f\,\big\|_{L^{p}(\R^d)}^{\frac{1}{2}}
\,\big\|\,|x|^{\gamma_3} f\,\big\|_{L^{q,\infty}(\R^d)}^{\frac{1}{2}}.
\end{equation*}
and
\begin{equation*}\label{WCKNr-weak}
\big\|\,|x|^{\gamma_1} f\,\big\|_{L^{s}(\R^d)}
\le C \big\|\,|x|^{\gamma_2} f\,\big\|_{L^{p}(\R^d)}^{\frac{1}{2}}
\big\|\,|x|^{\gamma_3}\nabla f\,\big\|_{L^{q,\infty}(\R^d)}^{\frac{1}{2}}.
\end{equation*}
\end{thm}

\begin{thm}\label{thm:main22}
Suppose $s, p, q, \gamma_1, \gamma_2, \gamma_3, \theta$ satisfy $s, p, q > 1$,
\begin{equation*}\label{WCKNr-2}
   \frac{1}{p} + \frac{\gamma_2}{d} < 1, \qquad \frac{1}{p} + \frac{\gamma_2-1}{d}\geq 0, \qquad
    \frac{1}{s} + \frac{\gamma_1}{d} = \frac12\Big(\frac{1}{p} + \frac{\gamma_2 - 1}{d}\Big) + \frac12\Big(\frac{1}{q} + \frac{\gamma_3}{d}\Big),
\end{equation*}
and
\begin{equation*}\label{WCKNr-4}
\gamma_1 \le \frac {\gamma_2}2 + \frac{\gamma_3}2.
\end{equation*}
Then there exists a constant $C> 0$ such that
\begin{equation*}\label{WCKNr-weak}
\big\|\,|x|^{\gamma_1} f\,\big\|_{L^{s,\infty}(\R^d)}
\le C \big\|\,|x|^{\gamma_2} f\,\big\|_{L^{p}(\R^d)}^{\frac{1}{2}}
\big\|\,|x|^{\gamma_3}\nabla f\,\big\|_{L^{q,\infty}(\R^d)}^{\frac{1}{2}}.
\end{equation*}
\end{thm}

\vspace{0.3cm}

To prove Theorems \ref{thm:main2} and \ref{thm:main22}, we must first establish the boundedness of bilinear fractional integrals.
Let $m\in \mathbb{Z}$ and $0<\alpha<md$, the multilinear fractional integral is defined by
$$\mathcal{I}_{\alpha,m}(\vec f)(x)=\int_{(\R^d)^m}\frac{f_1(y_1)\cdots f_m(y_m)}{\big(|x-y_1|^2+\cdots+|x-y_m|^2\big)^{\frac{md-\alpha}{2}}}dy_1\cdots dy_m.$$
Given a dyadic grid $\mathscr{D}$ and  a sparse family $\mathcal{S}$ in $\mathscr{D}$.
Define the dyadic fractional integrals over $\mathcal{S}$ by
$$
  \mathcal{I}_{\alpha,m}^{\mathcal{S}}(\vec{f})(x)
  =\sum_{Q\in\mathcal{S}}|Q|^{\frac{\alpha}{d}-m}
   \prod_{i=1}^m\int_Q f_i(y_i)dy_i \cdot\chi_Q(x).
$$
In \cite{M2009}, Moen proved that for any vector $\vec f$ consisting of non-negative functions, the following holds:
$$
  \mathcal{I}_{\alpha,m}(\vec{f})(x)
   \le C\sum_{Q\in\mathscr{D}}
   \frac{l(Q)^\alpha}{|3Q|^m}\int_{(3Q)^m} f_1(y_1)\cdots f_m(y_m)d\vec{y}\chi^{}_Q(x).
$$
There exist $2^d$ sparse families, $\{\mathcal{S}_t\}_{t\in \{0,1/3\}^d}$ such that
\begin{equation}\label{eq:q>1}
  \mathcal{I}_{\alpha,m}(\vec{f})(x)\lesssim
   \sum_{t\in\{0,1/3\}^d}(\mathcal{I}_{\alpha,m}^{\mathcal{S}_t})(\vec{f})(x).
\end{equation}

In order to prove Theorems \ref{thm:main2} and \ref{thm:main22}, the following two classes of weight functions are introduced.
\begin{definition}
Let $0 \le \alpha < d$, $1 < p_1, p_2 < \infty$, $0 < p, q < \infty$, and $\vec{p} = (p_1, p_2)$ with $\frac{1}{p} = \frac{1}{p_1} + \frac{1}{p_2}$.
A vector weight $(u, \vec{w}) = (u, w_1, \dots, w_m)$ is said to belong to the class $\mathcal{A}^{\alpha}_{\vec{p},q} := \mathcal{A}^{\alpha}_{\vec{p},q}(\R^d)$ if
\[
[\mu, \vec{w}]_{\mathcal{A}^{\alpha}_{\vec{p},q}}
 = \sup_{Q} \, |Q|^{\frac{\alpha}{d} + \frac{1}{q} - \frac{1}{p}}
 \|\mu\chi_{Q}\|_{L^{1}(\R^d)}^{\frac1q}
 \Biggl(\frac{1}{|Q|}\int_{Q} w_1(y_1)^{1-p'_1} \, dy_1\Biggr)^{\frac1{p'_1}}
 \|w_2^{-\frac1{p_2}} \chi_{Q}\|_{L^{p'_2,1}(\R^d)},
\]
where the supremum is taken over all cubes $Q \subset \R^d$, and $L^{p'_2,1}(\R^d)$ denotes the Lorentz space.
\end{definition}

\begin{definition}
Let $0 \le \alpha < d$, $1 < p_1, p_2 < \infty$, $0 < p, q < \infty$, and $\vec{p} = (p_1, p_2)$ with $\frac{1}{p} = \frac{1}{p_1} + \frac{1}{p_2}$.
A vector weight $(u, \vec{w}) = (u, w_1, \dots, w_m)$ is said to belong to the class $\mathcal{A}^{\alpha}_{\vec{p},q} := \mathcal{A}^{\alpha}_{\vec{p},q}(\R^d)$ if
\[
[\mu, \vec{w}]_{\mathcal{A}^{\alpha,*}_{\vec{p},q}}
 = \sup_{Q} \, |Q|^{\frac{\alpha}{d} + \frac{1}{q} - \frac{1}{p}}
 \|\mu\chi_{Q}\|_{L^{1, \infty}(\R^d)}^{\frac1q}
 \Biggl(\frac{1}{|Q|}\int_{Q} w_1(y_1)^{1-p'_1} \, dy_1\Biggr)^{\frac1{p'_1}}
 \|w_2^{-\frac1{p_2}} \chi_{Q}\|_{L^{p'_2,1}(\R^d)}.
\]
\end{definition}

\begin{thm}\label{thm:main3}
Let $0\leq \alpha<d$, $1<p_1,p_2<\infty$, $p_1\leq q<\infty$, $\vec{p}=(p_1,p_2)$ with $\frac{1}{p}=\frac{1}{p_1}+\frac{1}{p_2}$. If $(\mu, \vec w)\in \mathcal{A}^{\alpha}_{(p_1,p_2),q}$ with $\mu, w_1^{1-p'_1}\in A_{\infty}$, then there exists a constant $C> 0$ such that
\begin{equation*}
\|\mu^{\frac1q}\mathcal{I}_{\alpha,2}(f_1,f_2)\|_{L^{q}(\R^d)}
\leq C\|w_1^{\frac1{p_1}}f_1\|_{L^{p_1}(\R^d)}\|w_2^{\frac1{p_2}}f_2\|_{L^{p_2,\infty}(\R^d)}.
\end{equation*}
\end{thm}

\begin{proof}
Our approach is similar to that of \cite{WX}, although it should be noted that their method is restricted to a particular class of weight functions.

Set $\rho_1=(p_1'r_1)'$, where $r_1$ is the exponent in the sharp reverse H\"older inequality for the weights $w_1^{1-p'_1}\in A_{\infty}$. Similarly, set $\rho=(qr)'$, where $r$ is the exponent in the sharp reverse H\"older inequality for the weight $\mu\in A_\infty$.

Using duality and the definition of $\mathcal{I}_{\alpha,2}^{\mathcal{S}}$, we obtain
\begin{equation*}\label{lem-main1-eq1}
\begin{aligned}
&\|\mu^{\frac 1q}\mathcal{I}_{\alpha,2}^{\mathcal{S}}(f_1,f_2)\|_{L^q(\R^d)} \\
&= \sup_{\|h\|_{L^{q'}(\R^d)}=1} \Big|\int_{\R^d} \mathcal{I}_{\alpha,m}^{\mathcal{S}}(f_1,f_2)(x)h(x)\mu(x)^{\frac1q}\,dx\Big| \\
&= \sup_{\|h\|_{L^{q'}(\R^d)}=1} \sum_{Q\in\mathcal{S}} |Q|^{\frac{\alpha}{d}-2}
\Big(\prod_{i=1}^2 \int_Q |f_i(y_i)|\,dy_i\Big) \Big|\int_Q h(x)\mu(x)^{\frac1q}\,dx\Big|.
\end{aligned}
\end{equation*}
Applying H\"older's inequality to each integral yields
\[
\begin{aligned}
&\lesssim \sum_{Q\in S} |Q|^{\frac{\alpha}{d}+1}
\Big(\frac{1}{|Q|}\int_Q |h(x)|^{\rho}\,dx\Big)^{\frac{1}{\rho}}
\Big(\frac{1}{|Q|}\int_Q \mu(x)^{\frac{\rho'}{q}}\,dx\Big)^{\frac{1}{\rho'}} \\
&\quad \times  \Bigg[ \left( \frac{1}{|Q|} \int_Q w_1(y_1)^{\frac{\rho_1}{p_1}} |f_1(y_1)|^{\rho_1}\,dy_1 \right)^{\frac{1}{\rho_1}}
\left( \frac{1}{|Q|} \int_Q w_1(y_1)^{-\frac{\rho_1'}{p_1}}\,dy_1 \right)^{\frac{1}{\rho_1'}} \Bigg] \\
&\quad \times  \Bigg[ \frac{\| w_2^{\frac1{p_2}} f_2 \chi_Q\|_{L^{p_2,\infty}(\R^d)}}{|Q|^{\frac{1}{p_2}}}
\frac{\|w_2^{-\frac1{p_2}} \chi_Q\|_{L^{p_2',1}(\R^d)}}{|Q|^{\frac{1}{p_2'}}} \Bigg].
\end{aligned}
\]
Using the definitions of $\rho$ and $\rho_1$, we have
\[
\begin{aligned}
&\lesssim [\mu]_{A_\infty}^{\frac{1}{q}} [w_1^{1-p'_1}]_{A_\infty}
\sum_{Q\in S} |Q|^{\frac{\alpha}{d}+1}
\Big(\frac{1}{|Q|}\int_Q |h(x)|^{\rho}\,dx\Big)^{\frac{1}{\rho}}
\Big(\frac{1}{|Q|}\int_Q \mu(x)\,dx\Big)^{\frac{1}{q}} \\
&\quad \times \Bigg[ \left( \frac{1}{|Q|} \int_Q w_1(y_1)^{\frac{\rho_1}{p_1}} |f_1(y_1)|^{\rho_1}\,dy_1 \right)^{\frac{1}{\rho_1}}
\left( \frac{1}{|Q|} \int_Q w_1(y_1)^{1-p_1'}\,dy_1 \right)^{\frac{1}{p_1'}} \Bigg] \\
&\quad \times \Bigg[ \frac{\|w_2^{\frac{1}{p_2}} f_2 \chi_Q\|_{L^{p_2,\infty}(\R^d)}}{|Q|^{\frac{1}{p_2}}}
\frac{\|w_2^{-\frac1{p_2}} \chi_Q\|_{L^{p_2',1}(\R^d)}}{|Q|^{\frac{1}{p_2'}}} \Bigg].
\end{aligned}
\]
Note that $|Q| \leq 2|E(Q)|$, then
\[
\begin{aligned}
&\lesssim [\mu]_{A_\infty}^{\frac{1}{q}}[w_1^{1-p'_1}]_{A_\infty} [\mu, \vec{w}]_{A^{\alpha}_{(p_1,p_2),q}} \|w_2^{\frac1{p_2}} f_2\|_{L^{p_2,\infty}(\R^d)} \\
&\quad \times \sum_{Q\in S} |E(Q)|^{\frac{1}{q'}} \Big(\frac{1}{|Q|}\int_Q |h(x)|^{\rho}\,dx\Big)^{\frac{1}{\rho}} \\
&\quad \times |E(Q)|^{\frac{1}{p_1}} \left( \frac{1}{|Q|} \int_Q w_1(y_1)^{\frac{\rho_1}{p_1}} |f_1(y_1)|^{\rho_1}\,dy_1 \right)^{\frac{1}{\rho_1}}.
\end{aligned}
\]
Since $\frac{1}{p_1} + \frac{1}{q'} \geq 1$, we apply H\"older's inequality:
\[
\begin{aligned}
&\lesssim [\mu]_{A_\infty}^{\frac{1}{q}}[w_1^{1-p'_1}]_{A_\infty} [\mu, \vec{w}]_{A^{\alpha}_{(p_1,p_2),q}} \| w_2^{\frac1{p_2}}f_2\|_{L^{p_2,\infty}(\R^d)} \\
&\quad \times \Bigg\{ \sum_{Q\in S} |E(Q)| \Big[ \Big( \frac{1}{|Q|} \int_Q |h(x)|^{\rho}\,dx \Big)^{\frac{1}{\rho}} \Big]^{q'} \Bigg\}^{\frac{1}{q'}} \\
&\quad \times \Bigg\{ \sum_{Q\in S} |E(Q)| \Big[ \Big( \frac{1}{|Q|} \int_Q w_1(y_1)^{\frac{\rho_1}{p_1}} |f_1(y_1)|^{\rho_1}\,dy_1 \Big)^{\frac{1}{\rho_1}} \Big]^{p_1} \Bigg\}^{\frac{1}{p_1}}.
\end{aligned}
\]
By the boundeness of the maximal operator $M$, we conclude
\begin{align*}
&\|\mu^{\frac1q}\mathcal{I}_{\alpha,2}^{\mathcal{S}}(f_1,f_2)\|_{L^q(\R^d)} \\
&\lesssim \sup_{\|h\|_{L^{q'}(\R^d)}=1}\| w_2^{\frac1{p_2}}f_2\|_{L^{p_2,\infty}(\R^d)}
\, \|M(|h|^\rho)\|_{L^{\frac{q'}{\rho}}(\R^d)}^{\frac{1}{\rho}}
 \|M( w_i^{\frac{\rho_1}{p_1}}|f_1|^{\rho_1})\|_{L^{\frac{p_1}{\rho_1}}(\R^d)}^{\frac{1}{\rho_1}} \\
&\lesssim \| w_2^{\frac1{p_2}}f_2\|_{L^{p_2,\infty}(\R^d)} \| w_1^{\frac1{p_1}}f_1\|_{L^{p_1}(\R^d)}.
\end{align*}
The proof of the theorem is complete.
\end{proof}

\begin{proof}[Proof of Theorem~\ref{thm:main2}] By Lemma \ref{lem:power-2}, we have
\begin{equation}\label{thm:main2-eq1}
\||x|^{\gamma_1}\mathcal{I}_{\alpha,2}^{\mathcal{S}}(f_1,f_2)\|_{L^q(\R^d)}
\lesssim \||x|^{\gamma_2}f_1\|_{L^{p_1}(\R^d)} \||x|^{\gamma_3}f_2\|_{L^{p_2,\infty}(\R^d)}.
\end{equation}

Let $\nabla_{2n}$ be gradient in $\mathbb{R}^{2n}$. By the argument in \cite[p.125]{Stein}, one has
\begin{equation*}
\begin{aligned}
|f(x)g(x)|&\lesssim \int_{\mathbb{R}^{2n}}\frac{\nabla_{2n}\big(f(x-y)g(x-z)\big)}{\big(|y|+|z|\big)^{2n-1}}dydz\\
&\lesssim \mathcal{I}_{1,2}(|\nabla f|,|g|)(x)+\mathcal{I}_{1,2}(|f|,|\nabla g|)(x).
\end{aligned}
\end{equation*}
That is,
\begin{equation}\label{thm:main2-eq2}
|f(x)|^2\lesssim \mathcal{I}_{1,2}(|\nabla f|,|f|)(x).
\end{equation}
Combining the estimates \eqref{thm:main2-eq1} and \eqref{thm:main2-eq2}, the proof is completed.
\end{proof}
\begin{remark}
The proofs of Theorem \ref{thm:main2} and Theorem \ref{thm:main22} are similar, so we omit the proof of the latter.
\end{remark}

\section{Technical lemmas on power weights}\label{sec:appendix}

It is well known that $|x|^a \in A_p$ for $1 < p < \infty$ if and only if $-d < a < d(p-1)$, and $|x|^a \in A_\infty$ if and only if $-d < a < \infty$.

We now examine for which real numbers $a, b$ the inclusion $(|x|^a, |x|^b) \in A_{p,q}^{\alpha}$ (or $A_{p,q}^{\alpha,*}$) holds.

\begin{lemma}\label{lem:power-1}
Let $\alpha,a,b\in \R$ and $1\leq p, q<\infty$.
    For power weights \(\mu(x) = |x|^a\) and \(w(x) = |x|^b\), the following statements hold:

    \begin{enumerate}
        \item \((\mu, w) \in A_{p,q}^{\alpha}\) if and only if
        \[
        \alpha - \frac{d}{p} + \frac{d}{q} = \frac{b}{p} - \frac{a}{q} \ge 0,\qquad a > -d,
        \]
        and
        \[
        b < d(p-1)\;\; \text{for } p > 1, \qquad
        b \le 0 \;\; \text{for } p = 1.
        \]

        \item \((\mu, w) \in A_{p,q}^{\alpha,*}\) if and only if
        \[
        \alpha - \frac{d}{p} + \frac{d}{q} = \frac{b}{p} - \frac{a}{q} \ge 0, \qquad a \ge -d,
        \]
        and
        \[
        b < d(p-1)\;\; \text{for } p > 1, \qquad
        b \le 0 \;\; \text{for } p = 1.
        \]
    \end{enumerate}
\end{lemma}

\begin{proof}
We only prove the case $p = 1$ in (1) and the case $p > 1$ in (2); the remaining cases follow from similar arguments.

\textbf{Proof of (1) for $p=1$.}
Let $1 \le q < \infty$ and $\alpha \in \mathbb R$.
We need to characterise the triples $(a,b,\alpha)$ for which
\[
\sup_{Q} |Q|^{\frac{\alpha}{d}-1}
\Bigl(\int_{Q}|x|^{a}\,dx\Bigr)^{1/q}
\bigl\||\cdot|^{-b}\chi_{Q}\bigr\|_{L^{\infty}(\mathbb R^{d})} < \infty .
\]
Local integrability requires $a>-d$ and $b\le 0$.
It suffices to take the supremum over balls instead of cubes.
Following the idea in \cite{CGrafakos}, we split balls into two types.
A ball $B=B(x_{0},r)$ is of \emph{type I} if $2r<|x_{0}|$, and of \emph{type II} if $2r\ge |x_{0}|$.

\emph{Type I balls.}
For $x\in B$ we have $\frac12|x_{0}|\le |x|\le \frac32|x_{0}|$, hence $|x|\sim |x_{0}|$ on $B$.
Thus
\[
|B|^{\frac{\alpha}{d}-1}
\Bigl(\int_{B}|x|^{a}\,dx\Bigr)^{1/q}
\bigl\||\cdot|^{-b}\chi_{B}\bigr\|_{L^{\infty}}
\sim r^{\alpha-d+\frac{d}{q}}\,|x_{0}|^{\frac{a}{q}-b}.
\]
Fixing $x_{0}\neq0$ and letting $r\to0^{+}$ gives the necessary condition
\begin{equation}\label{eq:cond1}
\alpha-d+\frac{d}{q}\ge0 .
\end{equation}
Choosing $r=|x_{0}|/4$ yields the second condition
\begin{equation}\label{eq:cond2}
\alpha-d+\frac{d}{q}+\frac{a}{q}-b=0 .
\end{equation}
If both \eqref{eq:cond1} and \eqref{eq:cond2} hold, the supremum over type~I balls is finite.

\emph{Type II balls.}
For such a ball $B(x_{0},r)$ we have $B(x_{0},r)\subset B(0,3r)$; hence it is enough to consider balls centred at the origin.
For $B(0,R)$ we compute
\[
\begin{split}
&|B(0,R)|^{\frac{\alpha}{d}-1}
\Bigl(\int_{B(0,R)}|x|^{a}\,dx\Bigr)^{1/q}
\bigl\||\cdot|^{-b}\chi_{B(0,R)}\bigr\|_{L^{\infty}} \\
&\qquad\sim R^{\alpha-d}
\Bigl(\int_{0}^{R}\rho^{a+d-1}\,d\rho\Bigr)^{1/q}R^{-b}
\sim R^{\alpha-d+\frac{a+d}{q}-b}.
\end{split}
\]
This remains bounded for all $R>0$ precisely when \eqref{eq:cond2} holds.
Collecting the conditions obtained, we recover the stated characterization for $A_{1,q}^{\alpha}$.

\textbf{Proof of (2) for $p>1$.}
Now let $1<p,q<\infty$ and $\alpha\in\mathbb R$.
We study the condition
\[
\sup_{Q} |Q|^{\frac{\alpha}{d}-1}
\bigl\||\cdot|^{a}\chi_{Q}\bigr\|^{1/q}_{L^{1,\infty}(\mathbb R^{d})}
\Bigl(\int_{Q}|x|^{-\frac{bp'}{p}}\,dx\Bigr)^{1/p'} < \infty .
\]
Local integrability requires $a\ge -d$ and $b<d(p-1)$.
Again we may replace cubes by balls and split them into type~I and type~II.

\emph{Type I balls.}
Using $|x|\sim |x_{0}|$ on $B=B(x_{0},r)$ we obtain
\[
|B|^{\frac{\alpha}{d}-1}
\bigl\||\cdot|^{a}\chi_{B}\bigr\|^{1/q}_{L^{1,\infty}}
\Bigl(\int_{B}|x|^{-\frac{bp'}{p}}\,dx\Bigr)^{1/p'}
\sim r^{\alpha-d+\frac{d}{q}}\,|x_{0}|^{\frac{a}{q}-\frac{b}{p}} .
\]
Letting $r\to0^{+}$ with fixed $x_{0}\neq0$ gives
\begin{equation}\label{eq:cond1p}
\alpha-\frac{d}{p}+\frac{d}{q}\ge0 .
\end{equation}
Taking $r=|x_{0}|/4$ forces
\begin{equation}\label{eq:cond2p}
\alpha-\frac{d}{p}+\frac{d}{q}+\frac{a}{q}-\frac{b}{p}=0 .
\end{equation}

\emph{Type II balls.}
As before we reduce to balls $B(0,R)$.  A direct computation yields
\[
\begin{split}
&|B(0,R)|^{\frac{\alpha}{d}-1}
\bigl\||\cdot|^{a}\chi_{B(0,R)}\bigr\|^{1/q}_{L^{1,\infty}}
\Bigl(\int_{B(0,R)}|x|^{-\frac{bp'}{p}}\,dx\Bigr)^{1/p'} \sim R^{\alpha+\frac{a+d}{q}-\frac{b+d}{p}} .
\end{split}
\]
This is bounded for all $R>0$ exactly when \eqref{eq:cond2p} holds.
Together with the local integrability requirements, this gives the characterization of $A_{p,q}^{\alpha,*}$ for $p>1$.
\end{proof}

Using arguments similar to those given above, together with
\[
\|\chi_{B(x_0,r)}\|_{L^{p,1}(\R^d)}\;\sim\; r^{\frac{d}{p}},
\qquad
\bigl\|\,|\cdot|^{b}\chi_{B(0,r)}\bigr\|_{L^{p,1}(\R^d)}\;\sim\; r^{\frac{d}{p}+b}
\quad(b>-\tfrac{d}{p}),
\]
we obtain the following lemma.
\begin{lemma}\label{lem:power-2}
Let $\alpha,a,b_1,b_2\in \R$ and $1< p_1,p_2, q<\infty$.
    For power weights \(\mu(x) = |x|^a\), \(w_1(x) = |x|^{b_1}\) and \(w_2(x) = |x|^{b_2}\), , the following statements hold:

    \begin{enumerate}
        \item the weights \((\mu, (w_1,w_2)) \in A_{(p_1,p_2),q}^{\alpha}\) if and only if
        \[
        \alpha - \frac{d}{p_1}-\frac{d}{p_2} + \frac{d}{q} = \frac{b_1}{p_1}+\frac{b_2}{p_2} - \frac{a}{q} \ge 0,\quad a > -d,
        \quad b_1 < d(p_1-1), \quad b_2 < d(p_2-1).
        \]
      \item the weights \((\mu, (w_1,w_2)) \in A_{(p_1,p_2),q}^{\alpha,*}\) if and only if
        \[
        \alpha - \frac{d}{p_1}-\frac{d}{p_2} + \frac{d}{q} = \frac{b_1}{p_1}+\frac{b_2}{p_2} - \frac{a}{q} \ge 0,\quad a \geq -d,
        \quad b_1 < d(p_1-1), \quad b_2 < d(p_2-1).
        \]
          \end{enumerate}
\end{lemma}

\section*{Acknowledgments}

This work was partially supported by Key Scientific Project of Higher Education Institutions in Anhui Province (No.2023AH050145)
and the National Natural Science Foundation of China (NO. 12101010).

\end{document}